\documentclass[12pt]{article}
\usepackage[letterpaper,margin=0.85in]{geometry}

\usepackage{graphicx} 
\usepackage{amsmath,amssymb,amsfonts}
\usepackage[utf8]{inputenc}

\usepackage{todonotes}
\usepackage{natbib}
\newcommand{\tr}[1]{{\rm tr }\left(#1\right)}   
\newcommand{\h}{{\mathfrak{h}}}

\newcommand{\mean}[1]{\mathbb{E}\left (#1\right )}

\newcommand{\norm}[1]{\left\Vert#1\right\Vert}  
\newcommand{\abs}[1]{\left\vert#1\right\vert}   
\newcommand{\set}[1]{\left\{#1\right\}} 

\newcommand{\flind}[1]{\mathfrak{L}\kern-8pt{-}(#1)}
\newcommand{\preflind}[1]{\mathfrak{L}_*\kern-12pt{-}\;\;(#1)}

\newcommand{\1}{\mathbf{1}}
\newcommand{\vN}[1]{\mathcal{#1}}

\newcommand{\seque}[3]{({#1}_{#2})_{#2\in #3}}

\newcommand{\Z}[1]{\mathbb{Z}^{#1}}
\newcommand{\R}[1]{{\mathbb R}^{#1}}
\newcommand{\N}[1]{{\mathbb{N}}^{#1}}

\newtheorem{thm}{Theorem}

\newtheorem{cor}{Corollary}

\newtheorem{prop}{Proposition}
\newtheorem{defi}{Definition}
\newtheorem{rem}{Remark}
\newenvironment{pf}{\noindent{\it Proof.} }{\;\fbox{}\\}

\newcommand{\w}[1]{\mathbb{W}_{#1}}
\newcommand{\pol}[2]{\mathcal{P}^{#1}(#2)}
\newcommand{\sca}[3]{\langle #1,#2\rangle_{#3}}

\title{Fitness, predictability and equilibrium in biological dynamics\thanks{\textbf{Funding.} This research has been supported by ANID grant 13220168 \textit{Biological and quantum Open System Dynamics: evolution,
innovation and mathematical foundations.}}}

\author{%
Rolando Rebolledo\thanks{CIMFAV - Instituto de Ingenier\'ia Matem\'atica, Universidad de Valpara\'iso, Valpara\'iso, Chile. General Cruz 222, Valpara\'iso, CP 2340000, Chile. \texttt{rolando.rebolledo@uv.cl}. Corresponding author.}%
\and
Mauricio Tejo\thanks{CIMFAV - Instituto de Estad\'istica, Universidad de Valpara\'iso, Valpara\'iso, Chile. Gran Breta\~na 1111, Playa Ancha, Valpara\'iso, CP 2340000, Chile. \texttt{mauricio.tejo@uv.cl}.}%
\and
Mara A. Freilich\thanks{Division of Applied Mathematics and Department of Earth, Environmental, and Planetary Sciences, Brown University, 69 Brown St, Providence, RI 02912, USA. \texttt{mara\_freilich@brown.edu}.}%
\and
Pablo A. Marquet\thanks{Facultad de Ciencias Biol\'ogicas, Pontificia Universidad Cat\'olica de Chile, Santiago 8331150, Chile; The Santa Fe Institute, 1399 Hyde Park Road, Santa Fe, NM 87501, USA; Instituto de Sistemas Complejos de Valpara\'iso, Valpara\'iso, Chile; Wissenschaftskollege zu Berlin, Wallotstrasse 19, 14193 Berlin, Germany. \texttt{pmarquet@uc.cl}.}%
}

\begin{document}

\maketitle


\noindent \textbf{MSCcodes:}  
60G07,60H30,60F17,60J70,33C45,92B05,62M15.

\begin{abstract}
Understanding eco-evolutionary dynamics presents a number of mathematical challenges, among which the role of predictability and equilibrium in affecting eco-evolutionary dynamics is paramount. However, we are far from achieving an integral mathematical representation of the role of predictability, in part because the fundamental concept of fitness and its dynamics lack a general representation.
At present, Open System Dynamics theory provides a suitable methodological approach to address these challenges, starting from Stochastic Analysis, which appears to be one of their best contemporary representations. Here, we focus on the \textit{concept of fitness} and provide a mathematical theoretical framework to characterize the history of fitness change and the relationship between fitness and \textit{predictability}. 
We illustrate the use of this framework and explore the limiting behavior of biomass abundances modeled by stochastic processes in both discrete time and continuous time counterpart after rescaling (diffusion). In doing so, we focus on the consequences of neutrality in the search for equilibrium. In particular, the modeling of diffusions considers two kinds of basic settings: for relative abundances, our main processes take values on a $d$ dimensional simplex, where orthogonal polynomials played a central role in constructing such models; for non proportional abundances, these processes take values in the space $\mathbb{R}^d_+$ of vectors with positive components, in which we focus on two remarkable families of models. The paper concludes with statistical inference, presenting numerical examples for the case of relative abundances where orthogonal polynomials play a special role when estimating the relative abundance parameters under a maximum likelihood method. 
\end{abstract}

\maketitle

\section{Introduction}

The concept of fitness is fundamental to our understanding of ecological and evolutionary dynamics and is at the core of the principle of natural selection, as variability in fitness is mandatory for selection to occur \citep{orr2009fitness, wadgymar2024defining}. Fitness is associated with the information carried by different genetic variants, which is expressed in measurable phenotypic traits that can affect the performance of a given biological entity in a given environment. One of the difficulties of the concept is that it is nearly impossible to predict what the best combination of genes, and thus traits, is that maximizes fitness given a particular environment, implying that there is no way to measure fitness in absolute terms or as a distance from a reference or maximum attainable fitness value. The fitness concept is thus a relative one; relative to the performance of other individuals in the same environment, usually relative to the one with the highest fitness. The second issue is how to measure performance. Originally, and thanks to Herbert Spencer, natural selection was identified with the survival of the fittest, which led to circularity whereby the fittest are those that survive. This has led some authors to envision the entire theory of evolution by Natural Selection as tautological and metaphysical \cite{popper1978natural, manser1965concept}.  A more precise definition, originally emphasized by Darwin, was the number of descendants that each organism produced over its lifetime, but this is still circular: those organisms or types that achieve the greatest reproductive success do so because they are the ones that achieve the greatest reproductive success \citep{brandon1978adaptation}. This is also \textit{a posteriori} measurement of the ability of each organism to create a lineage that survives through time, which may be a practical metric for microbial populations but would be difficult to measure in long lived macroorganisms. Furthermore, it would be difficult to interpret since an organism may experience different environments during its lifetime \cite{jops2023life}, and the same number of offspring could imply different fitness values, and similarly for the age of the parents when producing offspring; fitness is not a fixed number or scalar, as highlighted by Van Valen \citep{VanValen}; it is a result of a process where spatial and temporal heterogeneity and stochasticity play an important role; Van Valen \citep{VanValen} argued that fitness is better represented as a vector, implicitly requiring knowledge of a history of interactions between the observed entity and its environment. The present contribution proposes an open dynamical system framework that takes this requirement seriously, by means of the general theory of processes.

Resolving these difficulties requires an \textit{a priori} way of assessing fitness. This, however, is impossible as it would require specifying the fitness value of all possible combinations of genes and phenotypes in a given environment. An influential response has been to turn to probabilities: treating fitness as a disposition to exhibit different reproductive outcomes with some probability in a given environment, i.e., a probability distribution over offspring number. This is the classical propensity interpretation of fitness \cite{mills1979propensity}, traceable to Fisher's claim for indeterminism \citep{fisher1934indeterminism,suarez2022complex}``If we take the view that all laws of natural causation are essentially laws of probability, we are only recognizing in theory what we have always known to be true in our daily affairs.''

The propensity interpretation, though not without its critics (see \cite{beatty1989rethinking,brandon1990adaptation,sober2001two,abrams2009unity, suarez2022complex}), inaugurated a probabilistic, albeit simple, view of fitness. Later attempts to specify a more complete probabilistic model 
of fitness have tended to be phenomenological \cite{pence2013new}, without identifying the drivers of the stochastic process and its mathematical representation. 

A separate, largely independent line of work has sought quantitative, model-based definitions of fitness for use directly in ecological and evolutionary dynamics. Adaptive dynamics theory identifies fitness with the dominant Lyapunov exponent of a phenotype's population-growth process in a given environment, and, for evolutionary analysis, evaluates its invasion fitness in the environment generated by a resident population at its ecological attractor \cite{metz1992should}. This approach is powerful precisely because it is model-dependent: fitness is defined only relative to a specified ecological scenario, which sidesteps rather than resolves the search for a fully general definition. Stochastic extensions of the Price equation take a complementary route, treating individual fitness itself as a random variable and showing that the entire distribution of fitness values, in addition to both its mean and its variance, contribute to directional evolutionary change \cite{rice2008stochastic}; this is close in spirit to the emphasis, developed in the following, on the interplay between the deterministic (closed or isolated system dynamics) and stochastic (open system dynamics) components of fitness, although it is built within a different mathematical formalism — the Price equation rather than the general theory of stochastic processes.

We aim to offers a definition of fitness that applies uniformly across levels of biological organization and across the discrete- and continuous-time settings, which is the gap we aim to fill here. This has important implications for model selection, parameter inference, and relating variance to \cite{azaele2006dynamical,rottjers2018hairballs,patterson2021and}. This task requires moving beyond the adaptive-dynamics and stochastic-Price-equation frameworks to a more general treatment of biological information in the context of stochastic processes.

Our aim, then, is to propose and justify a definition of fitness as a stochastic process driven by the ability of organisms — or, more generally, of a biomass flow — to predict their environment using the history of the process itself, i.e., the accumulated history of its interactions with the environment. To do this, we make use of the so-called General Theory of Processes \cite{meyer2006guide}, which is well suited to mathematically represent the dynamics of open systems. The analysis of these dynamics starts by considering stochastic processes as \textit{observable quantities}, while the \textit{state} of the given system corresponds to the joint probability distribution of those stochastic processes, defined on the set of all their trajectories. 

A terminological point is worth flagging at the outset to avoid an easy confusion. “Predictability” is used throughout this paper in a specific, well founded mathematical sense: a process is predictable if its value at time t is already known at time t-, i.e., it is measurable with respect to the history strictly before t (Section 2.1.3). This is not the same notion of “predictability” employed in the literature on bet-hedging and phenotypic switching, where the term refers instead to the statistical structure, the autocorrelation or the entropy rate of environmental fluctuations, and where a formal relationship has been established between an organism's long-term population growth rate and the amount of information available to it about a fluctuating environment \cite{donaldson2010fitness, kussell2005phenotypic,rivoire2011value}. A third, again distinct, sense of discounting-by-volatility appears in sequential Bayesian models of evidence accumulation in changing environments \cite{veliz2016stochastic}: there, an ideal observer's belief — a log-likelihood ratio — evolves as a nonlinear stochastic differential equation (SDE) whose drift is itself shown to equal a Kullback–Leibler divergence between the observation distributions of the candidate states, and whose nonlinear restoring term discounts older evidence at a rate set by the rate of environmental change, producing a nondegenerate stationary belief distribution where none would exist in a static environment. 

In what follows, we first provide a pedagogical introduction to key concepts in open systems and the general theory of processes (\S 2). We then derive treatment of fitness in biological system both continuous and discrete systems in the framework open system dynamics on the simplex in \S 3 before examining specific examples in \S 4. In \S 5 we demonstrate how this framework provides a new perspective on the parameter inference problem in population dynamics. We discuss the results in \S 6. 

\section{A short introduction to Open Systems and the  classical General Theory of Processes}\label{sec:fitness-in-GTP}
Within this section, we discuss a concept of fitness based on the general theory of processes, as stated in the foundational work of Claude Dellacherie \citep{Dellacherie-1972} and developed further in \citep{Dellacherie-Meyer}. The theory has been synthesized later by various authors (see, e.g. \citep{Protter}).

Motion is our fundamental principle for understanding life. Throughout history, humans have tried to mathematically represent motion through different approaches. The nineteenth century was particularly intrigued by an inexplicable motion discovered by the Botanist Robert Brown in 1827 \cite{brown1828xxvii}; the erratic motion of pollen in suspension in a liquid, which he demonstrated for many species, using live and dead (herbarium) individuals, and then extended to inorganic material of all sorts. These observations went unexplained for almost 80 years most likely because they were inconsistent with Newton's mechanics. It is well known that Einstein gave a suitable explanation of this phenomenon in 1905; previously, in his 1900 ``Doctorat d'État'' thesis on the Theory of Speculation, Bachelier  provided another interpretation inspired by market dynamics. Indeed, in addition to the fact that Einstein's work has shown the limits of Newtonian mechanics, the key idea in both approaches to Brownian motion was to go beyond the closed-system approach to understand natural phenomena, whereby the system components undergo deterministic dynamics. Similar conclusions were obtained in logic by  G\"odel with his celebrated result on the logical incompleteness of arithmetic (and thus of mathematics). Newton's method, consisting of isolating the motion of a part of nature to produce a (causal) theory of it, is at the heart of the closed-system approach paradigm. Therefore, motion resists isolation. This determines a mathematical challenge consisting of exploring the best way to represent open-system dynamics in nature \cite{Rebolledo:2019th}.

The basic motion structure is the flow: flow of biomass, flow of energy, electromagnetic waves, light, etc. The  interaction between flows corresponds to \textit{ events}. Events reflect the interaction of different types of motion in nature. To understand a given type of motion, one needs a suitable analytic methodology which could lead to appropriate mathematical modeling. That has been the challenge of Probability and Stochastic Analysis. 

The twentieth century provided various foundation proposals for Probability Theory. Among them, Kolmogorov's celebrated axiomatic theory based on measure theory (inherited from Lebesgue and Borel, among other mathematicians) was noteworthy. In parallel, von Neumann's work on the foundations of Quantum Mechanics introduced another theory of probability, a non-commutative one, that includes Kolmogorov's as a particular case. Stochastic processes then appeared as a representation of flows, or basic movements of energy, biomass, and any other form of motion. The contemporary discipline that studies their properties is now called stochastic analysis. We may reserve the adjective classical for the first version of stochastic analysis that uses Kolmogorov probabilities. This article aims to provide an interpretation of fitness within that framework.

The analysis of motion in an open system starts by distinguishing \textit{observables} from \textit{states}. An observable is a motion that we can transform and control directly, while the environment contains an undetermined number of different flows, some of which we become aware of (or detect) only indirectly. In stochastic jargon, the latter are called \textit{noise}, while the flow of observables defines the \textit{main dynamics}. The main dynamics and the noise are mathematically represented by stochastic processes. Therefore, both types of flows give rise to a stochastic process that describes the entire dynamics of the open system. The \textit{state of the system} is the probability distribution (in an infinite-dimensional space) of that process, and its distribution rules the interactions between different flows of the main system and the environment. As we explain below, in this context, fitness is a measure of the predictability of the motion based on its past history of interactions with the environment. In this article, our dynamic process will refer to the abundance of biomass, in which we will analyze its dynamics both in relative  and absolute  terms in detail in Sections \ref{sec:relative-abundance} and \ref{sec:ab-ab}, respectively.

\subsection{An overview of basic elements of the General Theory of Processes}
In this section, we will provide some basic definitions and outline the general stochastic framework in which our approach is embedded. The basic mathematical structure of the General Theory of Processes consists of a classical probability space $(\Omega,\vN{F},\mathbb{P})$ and a \textit{history} or \textit{filtration} $\mathbb{F}=\seque{\vN{F}}{t}{T}$ where $T$ is an ordered set (e.g. non-negative integers $\Z{}_+$ or $\N{}$, or non-negative real numbers $[0,\infty[=\R{}_+$), each $\vN{F}_t\subset\vN{F}$ is a $\sigma$ field and $\vN{F}_s\subset\vN{F}_t$ for $s,t\in T$ such that $s\leq t$. In Classical Probability Theory, events (flows interactions) are characterized as sets belonging to $\sigma$-fields. So, it is quite natural to consider increasing families of $\sigma$-fields mimicking a given history. It is also assumed that the history is continuous on the right, that is, $\vN{F}_t=\bigcap_{s>t}\vN{F}_s$, and any $\vN{F}_t$ contains all sets of probability zero for all $t$ (Dellacherie standard conditions). The regularity conditions of functions play an important role in this theory. Let us recall some basic facts about calculus. A real function $f$ has left-hand limits (respectively, right-hand limits) at $t\in\R{}$ if the limit $\lim_{s\to t; s<t}f(s)$, denoted $f(t-)$, (resp., $f(t+)=\lim_{s\to t; s>t}f(s)$ exists. A function is said to be continuous on the left (respectively, continuous on the right) if $f(t)=f(t-)$ (resp., $f(t)=f(t+)$). The function is said to be cadlag (from French ``continue à droite avec des limites à gauche'') if it satisfies $f(t)=f(t+)$ and $f(t-)$ exist. The set of all these functions is denoted $\mathbb{D}=D([0,\infty[,\R{})$. Similarly, the left-continuous functions are those for which $f(t)=f(t-)$ for all $t\in [0,\infty[$. A stochastic process $X:\Omega\times [0,\infty[\to \R{}$  is said to be \textit{adapted} to the history $\mathbb{F}=\seque{\vN{F}}{t}{[0,\infty[}$ if for each $t\in [0,\infty[$, the map $\omega\mapsto X(\omega,t)$ is $\vN{F}_t$ measurable. The relationship of an adapted process $X$ with $\mathbb{F}$ is particularly significant when modeling biomass flows, as $X$ represents the observed flow and $\mathbb{F}$ contains the information obtained through our observations of the process. Knowledge evolves through experiments, new observations with new adaptive innovations or more sophisticated instruments as is the case of humans, and the discovery of new biological events that are added at each step to the history $\mathbb{F}$. As a result, the main system model uses processes $X$ that are adapted to $\mathbb{F}$ and such that the trajectories $t\mapsto X(\omega,t)$ are in $\mathbb{D}$ for any $\omega\in\Omega$. In fact, the processes that represent the dynamics of flows in the main system have a richer structure. They are semimartingales that satisfy stochastic differential equations. A semimartingale $X$ is a stochastic process that admits a decomposition $X=F+M$ where $F$ is an adapted process with trajectories of finite variations\footnote{That is, they can be written as $F=B-D$, where $B$ and $D$ are two increasing processes.} in any finite interval of the real line, and $M$ is a local martingale (or noise) \citep{Dellacherie-Meyer,Protter}. Processes with finite variations represent what we directly measure through observations, whereas martingales are only indirectly observed; one needs an additional instrument to record interactions with those flows. In our case, they represent anything in the environment that could affect fitness and that we are not yet including in the representation of fitness. 
\subsubsection{Discrete time dynamics}
When $T=\N{}$, any adapted process $X=\seque{X}{n}{\N{}}$, such that $\sup_n\mathbb{E}(\abs{X_n})<\infty$ can be written in the following form:
\begin{equation}\label{decomposition-1}
    X_n=X_0+X^p_n+M_n,
\end{equation}
where $M$ is a martingale, that is, $\mathbb{E}(M_n|\vN{F}_{n-1})=M_{n-1}$, $M_0=0$ and $X^p$ is given by
\begin{equation}\label{decomposition-2}
    X^p_n=\sum_{k=1}^n\mathbb{E}\left(X_k-X_{k-1}|\vN{F}_{k-1}\right),
\end{equation}
for all $n\in\N{}$. 
Note that $X^p_n$ is $\vN{F}_{n-1}-measurable$. We say that $X^p$ is predictable at time $n$ with the information one has at time $n-1$. In practice, $X^p$ is observed and $X=X^p+M$, contains the noise or non-directly observed part $M$.

Let denote $\Delta X_k=X_k-X_{k-1}$, for all $k=1,\ldots, n$. So that $$X^p_n=\sum_{k=1}^n\mathbb{E}\left(\Delta X_k|\vN{F}_{k-1}\right).$$ 
Following a customary notation in Analysis, given any real number $a$, we write $a^+=\sup (a,0)$ its positive part, while $a^-=\sup(-a,0)$ is its negative part, so that $a=a^+-a^-$. Returning to \eqref{decomposition-2}, notice that we may write $X^p_n$ as a difference of two increasing processes, $X^p=B_n(X)-D_n(X)$, where
\begin{eqnarray}
    B_n(X)&=&\sum_{k=1}^n(\mean{\Delta X_k|\mathcal{F}_{k-1}})^+, \;\text{predictable birth}\\
    D_n(X)&=&\sum_{k=1}^n(\mean{\Delta X_k|\mathcal{F}_{k-1}})^-,\;\text{predictable death.}
\end{eqnarray}

 The process $X^p$ represents the flow of predictable birth and death of biomass based on previous observations, while the relationship of this flow with the unknown environment is represented by $M$. $M$ is indirectly detected through the increasing process $\langle M\rangle$, interpreted as the observed loss of energy or information, which also depends on the history of our knowledge.   The dissipation of energy measured under historical information $\mathbb{F}$ is
\begin{eqnarray}\label{volatility}
    \langle M\rangle_n&=&\sum_{k=1}^n\mathbb{E}\left((M_k-M_{k-1})^2|\vN{F}_{k-1}\right)\nonumber\\
    &=&\sum_{k=1}^n\left[\mathbb{E}\left ((\Delta X_k)^2|\vN{F}_{k-1}\right)-\left(\mathbb{E}(\Delta X_k|\vN{F}_{k-1}\right)^2\right],\\
    &=&\sum_{k=1}^n\left[\mathbb{E}\left ((\Delta X_k)^2|\vN{F}_{k-1}\right)-\left(\Delta B_k(X)-\Delta D_k(X)\right)^2\right]\nonumber
\end{eqnarray}
this is a predictable process, and it satisfies that $M^2-\langle M\rangle$ is a martingale. Incidentally, the process $\langle M\rangle$ has been denoted $\langle M,M\rangle$ and called ``predictable associated increasing process of the martingale $M$'' by Dellacherie and P.A. Meyer \cite{Dellacherie-Meyer}. We will prefer to call this process the\textit{ predictable energy dissipation (volatility) associated with the flow} $X$ and denote it by the symbol $V^p(X)=\langle X-X^p\rangle$.

Therefore, the couple of $\mathbb{F}$ predictable processes $(X^p,V^p(X))$ is essential to understand how the motion of biomass is adapted to its habitat or environment. That is, they are the basis for an open-system approach to the concept of fitness. They are predictable characteristics of biomass motion, according to the history of previous observations.

\subsubsection{Discrete time Markov chain}
To better understand the concept of predictability with respect to a given history $\mathbb{F}$, consider a simple yet general example of discrete time modeling. Assume that $X_n$ represents the value of a proportion of biomass at time $n$. So $X_n$ takes values in $[0,1]$. In addition, assume that $\seque{X}{n}{\N{}}$ is a Markov chain with values in a countable set $S\subset [0,1]$, say $S=\mathbb{Q}\cap [0,1]$ the rational numbers in $[0,1]$ and initial probability $q$. We write $q(A)=\sum_{x\in A}q_x$, where $q_x=q(\set{x})$. Let us denote $Q$ as its transition probability: 
\[Q(x,y)=\mathbb{P}(X_{n+1}=y\vert X_n=x).\] 
Given any measurable bounded function $f$ defined on $[0,1]$, let us denote $$Qf(x)=\sum_{y\in S}Q(x,y)f(y).$$ Moreover, given any probability measure $\mu$ in $[0,1]$, we define the action of $Q$ in $\mu$ as the measure $$\mu Q(A)=\sum_{x\in S}\sum_{y\in A}\mu_xQ(x,y).$$ So, given the initial probability $q$ of the Markov chain, one has $\mathbb{E}_q{f(X_1)=qQf}$. Also, if $R$ is another transition probability, we denote $RQ$ the transition probability $RQ(x,y)=\sum_{z\in S}R(x,z)Q(z,dy)$ (a product of matrices) . In particular, if $I$ denotes the identity function, $I(x)=x$, then $QI(x)=\sum_{y\in S}Q(x,y)y$. On the other hand, $Q^n(x,y)$ corresponds to
\[Q^n(x,y)=\mathbb{P}(X_n=y|X_0=x),\]
and, if the chain starts with an initial probability $q$,
\[qQ^nI=\mathbb{E}_q(X_n).\]
We adopt here the notation $\mathbb{E}_x$ in computations that involve conditional expectations when the chain starts with the initial probability $q=\delta_x$, the Dirac measure supported by $x\in S\subset [0,1]$.

Now consider the history generated by the observation of the process $X_n$ itself, that is, $\vN{F}_n=\sigma (X_0, X_1,\ldots,X_n)$, for any  $n\in\N{}$. Then, the computation of conditional expectations in terms of the transition probability $Q$ gives
\begin{eqnarray}
 \mathbb{E}_x\left(X_k-X_{k-1}|\vN{F}_{k-1}\right)&=&\mathbb{E}_x\left(X_k|\vN{F}_{k-1}\right)-X_{k-1}\nonumber\\
 &=&(Q-I)(X_{k-1}).
\end{eqnarray}
Call $L=Q-I$ the generator of the Markov chain; then, for any measurable bounded function $f$ defined on $[0,1]$, one obtains $Lf(y)=Qf(y)-f(y)$. In particular,
\begin{equation}
X^p_n=\sum_{k=1}^nL(X_{k-1}).
\end{equation}

Given any other bounded measurable function $g$ on $[0,1]$ we introduce the so called ``carré du champs'' operator
\begin{equation}
    \Gamma (f,g)=(Q-I)(fg)-f(Q-I)g-g(Q-I)f
\end{equation}

And for $f=g$, $\Gamma (f,f)=(Q-I)(f^2)-2f(Q-I)f=Qf^2-2fQf+f^2$, so that for $f(x)=x$, the formula \eqref{volatility} gives:
\begin{equation}
    V^p(X)_n=\sum_{k=1}^n\Gamma (X_{k-1},X_{k-1})=\sum_{k=1}^n\left(QX_{k-1}^2-2X_{k-1}QX_{k-1}+X_{k-1}^2\right).
\end{equation}
\subsubsection{Continuous time dynamics}
Going further, we summarize the concept of predictability in continuous time. Predictable processes have been studied in general for $T=[0,\infty[$ by Dellacherie in his fundamental book \citep{Dellacherie-1972}. A pedagogical treatment can be found in \citep{Protter} among many more recent books. Let us briefly recap this notion. Call $\mathbb{L}$ the set of all the left-continuous real functions, and consider in $\Omega\times [0,\infty[$ the $\sigma$--field $\vN{P}(\mathbb{F})$ generated by all processes $(\omega,t)\mapsto Z(\omega,t)$ that are adapted to history $\mathbb{F}$ and whose trajectories $t\mapsto Z(\omega,t)$ belong to $\mathbb{L}$. This is the \textit{predictable} field $\sigma$ in the product space $\Omega\times [0,\infty[$, and any process $Y$ satisfying the two-variable function $(\omega,t)\mapsto Y(\omega,t)$ is said to be $\vN{P}(\mathbb{F})$--measurable. In particular, all adapted continuous processes are predictable. However, this definition is too broad for our purposes since we want to predict facts or events that reflect the adaptability of the main system biomass; that is, we want to filter the noise. 

The important result that we will use in what follows is a very particular case of the existence of a predictable projection of an adapted process with finite variation on any finite real interval. Assume that $F$ is such a process (so it can be written as the difference of two increasing processes) and, in addition, that $\sup_t\mathbb{E}(\abs{F_t})<\infty$. Then there exists a unique predictable process $F^p$, such that $R=F-F^p$ is a martingale. So, if $X=F+N$ is a semi-martingale, with $F$ a process of finite variation and $N$ a martingale, $X$ can be written as $X=F^p+(N+R)$, where $F^p$ is predictable and $M=N+R$ is another representation of noise. The decomposition of semimartingales as the sum of a finite variation process $F$ and a martingale $M$, is not unique unless $F$ is predictable. Special semimartingales are those for which there exists a unique decomposition with a predictable process $F$  (see \cite{Dellacherie-Meyer}). That is the case for processes $X$ that are \textit{locally integrable}, which is particularly true for continuous, adapted processes; for instance, diffusion processes. Assume that $X$ is a special semimartingale, so it has a unique decomposition as $X=F+M$, where $F$ is a predictable process with finite variations. That is, $F$ as a finite-variation process can be written as the difference of two increasing processes: $F=B-D$, where $B$ represents (predictable) birth and $D$, (predictable) death. We denote this process by $X^p$. On the other hand, the noise $M$ in continuous time has no finite variations (it is not differentiable) it represents the interrelation of the observed main system with the environment and is only detected by the predictable loss of adequate physical magnitudes related to the phenomenon (energy, information). This dissipation is represented, as in the discrete time case, by the so-called \textit{associated increasing process of martingale} $M$, which we call $V^p(X)$, customarily denoted $\langle M\rangle$ (or $\langle M,M\rangle$), characterized by the property that $M^2-\langle M\rangle$ being a martingale. Finding the state (or probability distribution) of the noise has been a crucial problem in modeling open system dynamics by means of the GTP. In particular, the scientist working with discrete time data may easily identify birth and death of biomass, but it is a tough problem to characterize the distribution of the noise without including an additional hypothesis of Markovian behavior. Levy's theorem that characterizes Brownian motion through $V^p(X)=\langle M\rangle$ provides a key to solve that problem when one introduces time changes and uses limit theorems to build up a continuous-time model. This has been done through Central Limit Theorems (\cite{CLT-1980},\cite{CNS-CLT1980}) and the approximation of diffusion processes (\cite{Rebolledo:1979vk},\cite{Platen:1985vl}). 
\subsubsection{Defining fitness}
Here, we present our proposed definition of fitness in light of the General Theory of Stochastic Processes.

\begin{defi}
    The fitness of the biomass flow $X$ with respect to the history $\mathbb{F}$ is the couple of predictable processes $\Phi(X|\mathbb{F}):=(X^p,V^p(X))$, such that $X-X^p$ is a martingale. We call $X^p$ the \textit{predictable fitness} and $V^p(X)$ the \textit{predictable fitness volatility}.
\end{defi}
This structure allows us to model open dynamical systems under the assumption that observations do not perturb the object under study. Otherwise, one needs to use noncommutative probability theory. 

\subsection{Equilibrium and reversibility}
Another important concept in ecological models related to fitness concerns the notion of equilibrium and its difference from stationarity. Consider, for instance, a Markov chain with values in a countable set $S$ contained in $[0,1]$ with transition probability $Q$ as before. Since $S$ is countable, any probability measure $q$ on $S$ can be written as $q(A)=\sum_{x\in A}q(x)$ for all $A\subset S$, where $q(x)=q(\set{x})$, $x\in S$. One says that an initial probability $q$ is stationary or invariant if 
\begin{equation}\label{stationary}
    qQ=\sum_{x\in S}q(x)Q(x,y)f(y)=\sum_{x\in S}q(x)f(x),
\end{equation}
for all bounded measurable functions $f$. That is,
\[qL=0,\]
where $L(x,y)=Q(x,y)-I(x,y)$, $I(x,y)=\delta_{x,y}$.

This means that starting its motion with probability $q$, the observable $X_n$, at any time $n$, always has the same probability law $q$. We call this property the \textit{global balance}. 

Now, assume that $Q$ is a Markov transition matrix, and take any time $N>0$. The question is whether the reversed process $X^{\leftarrow}_n=X_{N-n}$ has the same probability law as that of $X_n$ for all $n\leq N$. Models of equilibrium statistical mechanics (like the Ising model simulated via Glauber dynamics) rely on reversible chains. Non-reversible chains are used to model \textit{non-equilibrium} systems (like fluid flow or driven diffusive systems). To achieve reversibility, one needs an additional property, stationarity is not sufficient, but it is necessary. We need a notion of \textit{local balance} of the Markov chain, which we illustrate below. 

So, from \eqref{stationary}$q$ is stationary if and only if 
\begin{equation}
    qQ(y)=\sum_{x\in S}q(x)Q(x,y)=q(y),
\end{equation}
for all $y\in S$. Since $x\mapsto q(x)$ is a function, one may apply $Q$ ``by the left'' to $q$ and one get:
\[Qq(y)=\sum_{x\in S}Q(y,x)q (x)=\sum_{x\in S}q(y)Q(y,x).\]

As a result, the chain is reversible if and only if it satisfies the detailed balance condition 
\begin{equation}
    q(x)Q(x,y)=q(y)Q(y,x),
\end{equation}
for all $x,y\in S$.
That is, intuitively, given two points $x,y\in S$, the probability of going to $y$ starting from $x$ with probability $q (x)$ is the same as reversing the motion from $y$ to $x$ starting with probability $q (y)$.

Finally, the net rate of probability flow -also called \textit{probability current}-between $x,y\in S$ is defined by the matrix $J=(J(x,y))_{x,y\in S}$:
\begin{equation}
    J(x,y)=q(x)Q(x,y)-q(y)Q(y,x),
\end{equation}
so that $q$ is an equilibrium probability if and only if $J=0$.
\subsection{Rescaling and limit behaviour}
It is well known that handling finite difference equations is harder than dealing with differential equations. The same is true in our ecological framework. So, as soon as observations provide a time series of data, one looks to implement a continuous time representation of it, even better if the whole dynamics are successfully written as a stochastic differential equation. To do that, one introduces a rescaling procedure and looks for the limit behavior of the rescaled biomass process.

We start with a discrete time representation of biomass evolution, $X=\seque{X}{n}{\N{}}$, with fitness $X^p$ and predictable fitness volatility $V^p(X)$. We rescale time through a couple of sequences $(\tau_N,\theta_N)_{N\in\N{}}$ where $\tau_N\in\N{}$ and $\theta_N:[0,\infty[\to\N{}$, so that we obtain a sequence of continuous time processes $t\mapsto Y_N(t)=\dfrac{1}{\tau_N}X_{\theta_N(t)}$, ($t\in [0,\infty[$). For instance, one may take $\theta_N(t)=[Nt]$ for all $0\leq t<\infty$, and $[\cdot]$ is the customary notation for the integer part of the real number $Nt$: $[Nt]:=\sup\set{n\in\N{}:\;n\leq Nt<(n+1)}$, and $\tau_N=N$, or $\tau_N=\sqrt{N}$. As a result, we get the following sequences of rescaled processes as follows:
\begin{eqnarray}
    Y_N(t)&:=&\frac{1}{\tau_N}X_{\theta_N(t)},\\
    Y^p_N(t)&:=&\frac{1}{\tau_N}X^p_{\theta_N(t)},\\
    &=&\frac{1}{\tau_N}\left(B_{\theta_N(t)}(X)-D_{\theta_N(t)}(X)\right)\nonumber,\\
    V^p_N(t)&:=&\frac{1}{\tau_N^2}V^p_{\theta_N(t)} (X).
\end{eqnarray}
We state below an easy proposition of a general limit theorem on rescaled processes proved in \cite{Rebolledo:1979vk}, Chapter III, Proposition 4 (see also \cite{Platen:1985vl}).
Consider $F(x,t)=\int_0^tf(x,s)ds$, and $B(x,t)=\int_0^tf^+(x,s)ds,\;D(x,t)=\int_0^tf^-(x,s)ds$, for $x\in\R{}$, $t\in [0,\infty[$. In addition, define $V(x,t)=\int_0^t v(x,s)ds$, where $v$ is a strictly positive function, for $(x,t)\in\R{}\times [0,\infty[$, $v(x,t)=\sigma^2(x,t)$.
\begin{prop}
    Assume that the sequence $X=\seque{X}{n}{\N{}}$ satisfies
    \begin{equation}
       \sup_{n\in\N{}} \mathbb{E}\left(\abs{X_n-X_{n-1}}^2\right)<\infty.
    \end{equation}
  Under the following hypotheses,
    \begin{itemize}
    \item[(H0)] $Y_N(0)$ converges in distribution to $x\in\R{}$;
\item [(H1)] $\abs{\frac{1}{\tau_N}B_{\theta_N(t)}(X)-B(Y_N(t),t)}\to 0$ in probability for all $t\geq 0$ as $N\to\infty$;
\item [(H2)] $\abs{\frac{1}{\tau_N}D_{\theta_N(t)}(X)-D(Y_N(t),t)}\to 0$ in probability for all $t\geq 0$ as $N\to\infty$;
\item [(H3)] $\abs{\frac{1}{\tau_N^2}V^p_{\theta_N(t)}(X)-V(Y_N(t),t)}\to 0$ in probability for all $t\geq 0$ as $N\to\infty$;
    \end{itemize}
    the rescaled process $Y_N$ converges in distribution to a diffusion process $X$ that satisfies the stochastic differential equation
    \begin{equation}\label{limit-diff}
        X(t)=x+\int_0^tf(X(s),s)ds+\int_0^t\sqrt{v(X(s),s)}dW_s,
    \end{equation}
    where $W$ denotes a canonical Brownian motion on the real line.
\end{prop}
Thus, this Proposition shows that rescaled biomasses converge in distribution to a continuous semimartingale (a diffusion process) if the pair of fitness and predictable volatility of the fitness converges to the respective fitness and volatility of the limit process. 

Now, the diffusion that represents the biomass data may provide the absolute abundance of a given species. But often, especially when there is a community with $d$ different species, genomic data, or different carrying capacities in different locations, working with proportions is more adequate; that is, the dynamics of relative abundances. We consider below the representation of open system dynamics of relative abundances.

 \section{The flow of relative abundances}\label{sec:relative-abundance}
 A proportion of biomass corresponds to an observable with values in the bounded real interval $[0,1]$. Consider $d$ biological entities with biomass proportions $x_i\geq 0$ ($i=1,\ldots,d$) normalized to have $\sum_{i=1}^dx_i\leq 1$. 

Therefore, the stochastic process $X=(X_1,\ldots,X_d)$ that represents the flow of the entire biomass proportion takes values in the simplex of dimensions $d$.

\subsection{Continuous Open System Dynamics in the simplex}\label{sec:OSD-S}

We start by introducing some  notation. Consider the $d$-dimensional vector space of real numbers, for each vector $x=(x_1\ldots,x_d)$ we write $\norm{x}=\sum_{i=1}^d\abs{x_i}$. We denote $T^d$ the $d$-dimensional closed simplex:
\[T^d=\set{x\in\R{d}:\;x_i\geq 0,\;i=1,\ldots,d,\;\norm{x}\leq 1},\]
and we reserve the notation $S^d$ for the extremal points:
\[S^d=\set{x\in T^d:\;\norm{x}=1}.\]

$X^p$ is a predictable process with trajectories in the simplex $T^d$. Thus, we now focus on the open system dynamics in the simplex to address the following problems: \begin{enumerate}
    \item How to compute fitness as a function of $X$ in continuous time;
    \item Reversibility and equilibrium;
    \item Statistical inference to estimate fitness.
\end{enumerate}

The set $T^d$ is commonly endowed with the Borel $\sigma$ field $\vN{B}(T^d)$ generated by the open sets therein. We introduce a Hilbert space $\h=L^2(T^d,\vN{B}(T^d),\mu)$ considering a probability measure $\mu$ on the measurable space $(T^d,\vN{B}(T^d))$. We are interested in probabilities that are absolutely continuous with respect to the Lebesgue measure, say such that $\mu(dx_1,\ldots,dx_d)=w(x_1,\ldots,x_d)dx_1\ldots dx_d$, that is, $w$ is the density or Radon-Nikodym derivative of $\mu$ with respect to the Lebesgue measure. In this case, the orthonormal basis of $\h$ can be obtained by orthogonal polynomials.

As in \cite{marquet2017proportional}, we will assume that the total biomass of the system is divided in individuals that are
asymptotically independent and equivalent in fitness, as envisioned in Neutral Theory \cite{hubbell2011unified}. This, in part, might be reflected in our model of relative abundances when choosing  a weight function $w$ of product type, as we do below, since it could suggest that the original species abundances share similar fitnesses and behave independently, as shown in Section \ref{sec:CIR}. 

We follow \citep{Aktas} and \citep{Dunkl-Xu} to introduce a suitable Hilbert space for our analysis. Consider a sequence $\gamma=(\gamma_1,\ldots,\gamma_{d+1})\in]-1,\infty[^{d+1}$, and define a weight function $\w{\gamma}$ as follows:
\begin{equation}\label{weight-function}
    \w{\gamma}(x)=x_1^{\gamma_1}\ldots x_d^{\gamma_d}(1-\norm{x})^{\gamma_{d+1}},
\end{equation}
for all $x=(x_1,\ldots,x_d)\in T^d$.

Let us denote $C_\gamma=1/ {\int_{T^d}\w{\gamma}(x)dx}$, so that $\mu_\gamma (dx)=C_\gamma\w{\gamma}(x)dx$ is a simplex probability measure, absolutely continuous with respect to the Lebesgue measure. We then consider the Hilbert space $\h (\gamma)=L^2(T^d,\mu_{\gamma})$ with the scalar product:
\begin{equation}
 \sca{f}{g}{\gamma}=\int_{T^d}fgd\mu_\gamma  . 
\end{equation}

Remarkable orthonormal bases for $\h(\gamma)$ are obtained using orthogonal polynomials. The properties of such polynomials have been thoroughly studied in the reference book \citep{Dunkl-Xu}. We call $\pol{n}{\gamma}$ the set of all orthogonal polynomials of degree $n$ in $\h (\gamma)$.

We will enforce the following theorem proved by Dunkl and Xu in \citep{Dunkl-Xu}. Let $\1$ denote the vector with all coordinates equal to 1.
\begin{thm}\label{Th_Dunkl-Xu}
    Any polynomial $P\in\pol{n}{\gamma}$ is an eigenvector of a second-order differential operator $L_\gamma$ of diffusion type, that is
    \begin{eqnarray}
        L_\gamma P&=&\sum_{i=1}^d x_i(1-x_i)\frac{\partial^2 P}{\partial x_i^2}\nonumber\\
        &-&2\sum_{1\leq i<j\leq d}x_ix_j\frac{\partial^2 P}{\partial x_i\partial x_j}\label{dunkl}\\
        &+&\sum_{i=1}^d \left (\gamma_i+1-(\langle\gamma,\1\rangle+d+1)x_i\right)\frac{\partial P}{\partial x_i}\nonumber\\
        &=&-n(n+\langle\gamma,\1\rangle+d)P.\nonumber
    \end{eqnarray}
\end{thm}

Let us denote $\Omega=C([0,\infty[,T^d)$ the space of continuous functions defined in the interval $[0,\infty[$ with values in $T^d$. Call $X(t,\omega)=\omega (t)$, for all $\omega\in \Omega$, $t\in [0,\infty[$ the canonical process, and consider the family of $\sigma$ fields $\mathbb{F}=(\vN{F})_{t\geq 0}$, where $\vN{F}_t=\sigma (X(s):\;s\leq t)$ for all $t\geq 0$ and $\vN{F}=\sigma (X(t):\;t \geq 0)$. We recall that the foundational work of Stroock and Varadhan characterizes diffusions as follows\cite{w.stroock1969diffusion, w.stroock1969diffusion, chetrite2021e}. The continuous process $X$ is a diffusion associated with the second-order differential operator $L_\gamma$ if and only if there exists a unique probability $\mathbb{P}_\gamma$ such that for all functions of two-valued real $g$ defined on the simplex, we have the process
\begin{equation}\label{mart-problem}
    M^g(t):=g(X(t))-g(X(0))-\int_0^tL_\gamma g(X(s))ds, \;(t\geq 0),
\end{equation}
is a $(\mathbb{F},\mathbb{P}_\gamma)$-martingale. One says that $\mathbb{P}_\gamma$ is a (unique) solution to the \textit{Martingale Problem} $(\nu,L_\gamma)$, where $\nu$ is the initial probability law of the random variable $X(0)$. The martingale $M^g$ has the associated increasing process
\begin{equation}\label{carre1}
    \langle M^g\rangle=\int_0^t\Gamma  g(X(s))ds,
\end{equation}
where $\Gamma$ is the carré du champ operator defined as
\begin{equation}\label{carre}
    \Gamma g=L_\gamma g^2-2fL_\gamma g,
\end{equation}
for any twice differentiable function $g$.

We write the generator \eqref{dunkl} in condensed form as follows:
\begin{equation}\label{generator-resum}
    L_\gamma g(x)= f(x)^T\nabla g(x)+ \frac{1}{2}(A(x)\nabla)^T\nabla f(x),        
\end{equation}
where $f(x)$ is the column vector with components $f_i(x)=\gamma_i+1-(\langle \gamma+\1,\1\rangle)x_i$, and $A(x)$ is the positive definite matrix with components
\[a_{i,j}(x)=-2x_ix_j,\; i\not=j;\; a_{i,i}=2x_i(1-x_i).\]
Call $\sigma(x)$ the square root of $A(x)$, $\sigma(x)\sigma(x)^T=A(x)$.
These notations allow us to write the stochastic differential equation associated with the generator $L_\gamma$ according to the Stroock-Varadhan theory.

\begin{cor}
Under the probability $\mathbb{P}_\gamma$, the canonical process $X$ introduced here before satisfies a stochastic differential equation  of the form 
\begin{equation}\label{diffusion}
X(t)=X(0)+\int_0^tf(X(s))ds+\int_0^t \sigma (X(s))dW(s),
\end{equation}
where $W$ is a $d$-dimensional Brownian motion. Call $M(t)=\int_0^t\sigma (X(s))dW(s)$, $t$ running on $[0,\infty[$.

As a result, the fitness associated with the biomass flow represented by $X$ is given by the process $X^p(t)=\int_0^tf(X(s))ds$ and the fitness volatility is $\langle M\rangle (t)=\int_0^t\tr{A(X(s))}ds$, for all $t\geq 0$.

\end{cor}

We explore some useful relations between orthogonal polynomials and the biomass diffusion process $X$ defined by $L_\gamma$.

\begin{cor}
    Given any polynomial $P\in\pol{n}{\gamma}$, the process \begin{equation}
     M^{P}(t)=P(X(t))-P(X(0))+n(n+\langle\gamma,\1\rangle+d)\int_0^t  P(X(s))ds,
    \end{equation}
    is a continuous martingale that satisfies
    \begin{equation}
        \langle M^P\rangle (t)=\int_0^t\tr{A(X(s))\nabla\nabla^TP(X(s)}ds.    \end{equation}
\end{cor}
\begin{pf}
    By \eqref{mart-problem}, $M^P(t)=P(X(t))-P(X(0))-\int_0^tL_\gamma P(X(s))ds$, ($t\geq 0$) is a martingale. Therefore, it suffices to apply that $P$ is an eigenfunction of $L_\gamma$ and the stated result follows. Moreover, a straightforward computation of $\Gamma P$ gives:
    \begin{eqnarray}
        \Gamma P&=&L_\gamma P^2-2P L_\gamma P\nonumber\\
        &=&2\left[\sum_{i=1}^dx_i(1-x_i)\left(\frac{\partial P}{\partial x_i}\right)^2-2\sum_{1\leq i<j\leq d}x_ix_j\frac{\partial P}{\partial x_i}\frac{\partial P}{\partial x_j}\right]\nonumber\\
        &=&\tr{A\nabla P\nabla^TP},
    \end{eqnarray}
    which yields the expression for $\langle M^P\rangle$.
\end{pf}

Notice that for a polynomial $P_1$ of degree 1 and if $X(0)=0$, then
\[P_1(X(t))-P_1(0)=\sum_{i=1}^d\frac{\partial P_1}{\partial x_i}(0)X_i(t)=\langle \nabla P_1(0),X(t)\rangle,\]
and 
\[\langle \nabla P_1(0),X(t)\rangle-(1+\langle \gamma,\1\rangle+d)\int_0^t  P_1(X(s))ds,\]
is a continuous martingale.
\begin{cor}
 For any $d$, the fitness of a biomass flow moving on $T^d$ endowed with the measure $\mu_\gamma$ and starting from 0, is given by the couple $(X^p,\langle M^{P_1}\rangle)$ that satisfies the following equations:
 \begin{eqnarray}
    \langle \nabla P_1(0),X^p(t)\rangle&=&-(1+\langle \gamma,\1\rangle+d)\int_0^tP_1(X(s))ds\\
    &=&\int_0^t\left(\sum_{i=1}^d
(\gamma_i+1-(\langle \gamma,\1\rangle+d+1)X_i(s)\right)\frac{\partial}{\partial x_i}P_1(0)ds,\nonumber\end{eqnarray}
\begin{equation}
    \langle M^{P_1}(t)\rangle=\int_0^t\sigma^2(X(s))ds,
    \end{equation}
where $\sigma (X(s))$ is the diagonal matrix with components $c_P\sqrt{X_i(s)(1-X_i(s))}$, $i=1,\ldots,d$, and $P_1$ is an orthogonal polynomial of degree 1 in $d$ variables, $P_1(x)=\sum_{i=1}^d(r_i+s_ix_i)$, $x\in T^d$.
\end{cor}
\begin{pf}
Let an orthogonal polynomial $P_1\in\vN{P}^1$ be given. So, $L_\gamma P_1$ reduces to
\[L_\gamma P_1(x)=\sum_{i=1}^d
(\gamma_i+1-(\langle \gamma,\1\rangle+d+1)s_ix_i=-(1+\langle \gamma,\1\rangle+d)P_1(x),\]
for all $x\in T^d$. Therefore, an application of \eqref{mart-problem}, noticing that $\nabla P_1(x)=\nabla P_1(0)$, shows that
\begin{equation}\label{mart1}
M^{P_1}(t)=P_1(X(t))-P_1(0)-\int_0^t\langle b(X(s)),\nabla P_1(0)\rangle ds,\end{equation}
is a continuous martingale, which can be written in the form
\begin{equation}\label{mart2}
M^{P_1}(t)=P_1(X(t))-P_1(0)+(1+\langle \gamma,\1\rangle+d)\int_0^tP_1(X(s))ds\end{equation}
\[=\langle \nabla P_1(0),X(t)\rangle+(1+\langle \gamma,\1\rangle+d)\int_0^tP_1(X(s))ds.\]

Consider the stochastic differential equation associated to $L_\gamma$,
\begin{equation}\label{SDE}
    dX(t)=f(X(t))dt+\sigma(X(t))dW(t),
\end{equation}
where $\sigma (x)$ is a $d\times d$-matrix such that $\sigma(x)\sigma(x)^T=A(x)$ and $f$ are defined in \eqref{generator-resum}.

Consider the function $g(x)=\langle \nabla P_1(0),x\rangle$, $x\in T^d$, this function satisfies: $\nabla g(x)=\nabla P_1(0)$, $\nabla\nabla^T g(x)=0$. So, a straightforward application of Ito's formula gives:

\begin{equation}\label{mart3}
    \langle \nabla P_1(0), X(t)\rangle=\int_0^t\langle f(X(s)),\nabla P_1(0)\rangle ds+\int_0^t\langle \nabla P_1(0),\sigma (X(s))dW(s)\rangle.\end{equation}

Notice that $\langle \nabla P_1(0),X(t)\rangle=P_1(X(t))-P_1(0)$, so that \eqref{mart3} combined with \eqref{mart1} allows us to identify $M^{P_1}$ as
\begin{equation}
    M^{P_1}(t)=\int_0^t\langle \nabla P_1(0),\sigma(X(s))dW(s)\rangle.
\end{equation}
Moreover, by \eqref{mart2}, we get
\begin{equation}
\langle \nabla P_1(0),X(t)\rangle=\int_0^t\langle \nabla P_1(0),\sigma (X(s))dW(s)\rangle-(1+\langle \gamma,\1\rangle+d)\int_0^tP_1(X(s))ds.
\end{equation}

Therefore,
\begin{equation}
\langle \nabla P_1(0),X^p(t)\rangle=-(1+\langle \gamma,\1\rangle+d)\int_0^tP_1(X(s))ds.   
\end{equation}
\end{pf}

\begin{rem}
    Let us rephrase the previous corollary. If one assumes that neutrality suitably describes the large-time behavior of biomass flows, then fitness is an affine function of the biomass, while fitness volatility  is quadratic.
\end{rem}
\subsection{Invariant measures, entropy production and equilibrium}
A natural question that arises is the existence of invariant measures for the dynamics defined by the Markov semigroup with generator $L_\gamma$. Under the neutrality hypothesis represented at the level of the measure-theoretic structure imposed on the simplex by the product measure $\mu_\gamma$, one looks for an invariant probability measure $\nu (dx)=\rho (x)dx$ defined on the simplex. That is, $\rho$ is a probability density with respect to the Lebesgue measure that satisfies $\int_{T^d} \rho(x)dx=1$ and should satisfy the equation
\begin{equation}
    L_{\gamma *}\rho=0,
\end{equation}
where $L_{\gamma *}$ is the predual generator of the semigroup. This predual generator is obtained by duality as follows
\begin{equation}
    \langle L_{\gamma *}\rho,g\rangle_\gamma=\langle \rho,L_\gamma g\rangle_\gamma,
\end{equation}
for all twice differentiable functions $g$ defined in $T^d$. The above definition requires that $\rho$ be equally twice differentiable, and the explicit expression of $L_{\gamma *}\rho$ is

\begin{eqnarray}
       ( L_{\gamma *}\rho)(x)&=&\frac{1}{2}\nabla^T\left(\nabla (A(x)\rho(x))-\nabla^T(f(x)\rho(x))\right)\nonumber\\
        &=&\nabla^T\left(\frac{1}{2}\nabla (A(x)\rho(x))-f(x)\rho(x)\right)\\
        &=&\frac{1}{2}\sum_{i,j=1}^{d}\frac{\partial^2(a_{i,j}(x)\rho(x))}{\partial x_i\partial x_j}-\sum_{i=1}^{d}\frac{\partial(f_i(x)\rho(x))}{\partial x_i}.\nonumber
    \end{eqnarray}

Going further, we follow \cite{Jiang.2004} to include in our analysis the expressions of \textit{entropy production} and \textit{probability current} that allow us to characterize equilibrium. The probability current, also called flux, is defined as a functional $J(\rho,x)$ as follows:
\begin{equation}\label{current}
J(\rho,x)= f(x)\rho(x)-\frac{1}{2}\nabla (A(x)\rho(x)),   
\end{equation}
for a probability density $\rho\in L^1(T^d,\mu_\gamma)$ and $x\in T^d$. 

Consider the diffusion process $X$ and call $\rho_t(x)$ the probability density of the random variable $X_t$, ($t\geq 0$). The Fokker-Planck equation satisfied by $\rho_t$ can be written as:
\begin{equation}
    \frac{\partial \rho_t(x)}{\partial t}=-\nabla^TJ(\rho_t,x).
\end{equation}

The probability current (or flux) $J$ connects three concepts for a stationary measure: equilibrium, reversibility, and entropy production. Under a given stationary measure, reversibility concerns the possibility that the biomass process could be reversed in time; the entropy production controls reversibility by computing the derivative of the relative entropy between the laws of the forward and reverse processes, given a stationary measure (see \cite{Jiang.2004}). If this derivative is null, then the relative entropy remains constant. An equilibrium stationary measure is characterized by zero entropy production. Following \cite{Jiang.2004} Theorem 4.1.7, the expression for entropy production in our notation is given by:
\begin{equation}\label{entprod}
    e_p=\int_{\R{d}}\frac{\left(J(\rho,x)\right)^TA^{-1}(x)J(\rho,x)}{\rho(x)}dx.
\end{equation}

Since our process takes values in the simplex $T^d$, the probability current is null outside this domain. Therefore, a stationary density $\rho$ is an equilibrium density if and only if $J(\rho,x)=0$ for all $x\in T^d$; that is, if and only if $J$ satisfies the \textit{detailed balance condition}, and so the process is reversible. 

Invariant measures that arise in these open abundance dynamics naturally assume a Gibbs-type structure due to the ordinary differential equation $J(\rho,x)=0$ in the simplex. The associated generators are symmetric with respect to these invariant measures, yielding spectral decompositions in terms of orthogonal polynomial systems, particularly Jacobi polynomials in the reversible simplex case.

Let us define the class of Gibbs stationary densities on $\R{d}$ endowed with the $d$-Lebesgue measure as 
\begin{equation}
\vN{G}(\R{d})=\set{\rho\in L^1(\R{d},\lambda_d): \rho(x)=\frac{1}{Z}\exp (U(x)), \;(x\in\R{d}), Z=\int_{\R{}}\exp (U(x))dx<\infty} ,   
\end{equation}
The function $U$ will be referred to as the \textit{fitness potential} of the biomass dynamics if $\rho(x)=\frac{1}{Z}\exp (U(x))$ is an equilibrium density for its Markov generator. 

Here below, we consider the family $\vN{G}(T^d)$ of Gibbsian densities defined on the $d$-dimensional simplex endowed with the Lebesgue measure.
\begin{thm}\label{th:dirch-inv}
The density $\rho(x)=C_\gamma\mathbb{W}_\gamma(x)$, of the measure $\mu_\gamma$ is the unique equilibrium density of the biomass dynamics. It is of Gibbsian type with fitness potential $U$ whose gradient is
\begin{equation}\label{fitness-pot}
\nabla U(x)=A(x)^{-1}(2f(x)-\nabla A(x))=\sum_{i=1}^d\gamma_i\log x_i+\gamma_{d+1}\log(1-\norm{x}),
\end{equation}
for all $x\in T^d$.

Due to the form of the potential $U$, the equilibrium measure $\mu_\gamma$ has the structure of a Dirichlet measure. 

Notice that the equilibrium measure $\mu_\gamma$ is Dirichlet if and only if any of the equivalent conditions are satisfied 
\begin{equation}\label{condition-U}
    \log \mathbb{W}_\gamma(x)+\log C_\gamma=A(x)^{-1}(2f(x)-\nabla A(x)) \;(x\in T^d)
\end{equation}
\begin{equation}\label{condition-current}
 J(\rho,x)=0 \;(x\in T^d),   
\end{equation}
where $\rho(x)=C_\gamma \mathbb{W}_\gamma(x)$
This means that the Dirichlet measure $\mu_\gamma$ is the equilibrium measure if and only if fitness is an affine function and fitness volatility is quadratic.
\end{thm}

\begin{rem}
    The previous theorem has an extension in the more general framework of diffusion processes on $\R{d}$ with values in an open domain $D\subset\R{d}$ with fitness $f$ and volatility $A$. There exists an equilibrium density $\rho\in\vN{G}$ if and only if \eqref{condition-U} is satisfied. As a result, if $f$ is a linear or affine function, the equilibrium is obtained with a quadratic volatility satisfying \eqref{condition-U}. 
    
\end{rem}
\subsection{The one-dimensional case}

Assume first, as an example, that the biomass flow is represented by a one-dimensional diffusion process. The simplex $T^1$ is the interval $[0,1]$, the weight is $w_{\alpha,\beta}(x)=(1-x)^\alpha x^\beta$, $\alpha,\beta>-1$, and we consider \textit{Jacobi orthogonal polynomials} defined by:
\begin{equation}
    J_n^{\alpha,\beta}(x):=(1-x)^{-\alpha}x^{-\beta}\frac{d^n}{dx^n}\left [(1-x)^{n+\alpha}x^{n+\beta}\right].
\end{equation}
For $n=1$, we obtain $J_1^{\alpha,\beta}(x)=\beta+1-(\alpha+\beta+2)x$.

 In particular, assuming that $X_0=0$, for $J_1^{\alpha,\beta}$ we obtain
\[M^{J^{\alpha,\beta}_1}_t=-(\alpha+\beta+2)X_t+\int_0^t(\alpha+\beta+2)J^{\alpha,\beta}_1(X_s)ds.\]
As a result, simplifying the constant $(\alpha+\beta+2)$ we obtain that $X_t-\int_0^tJ_1^{\alpha,\beta}(X_s)ds$ is a martingale $M=\frac{M^{J^{\alpha,\beta}_1}}{\alpha+\beta+2}$ and so
\begin{equation}\label{ifitness}
    X^p_t=\int_0^tJ_1^{\alpha,\beta}(X_s)ds,\;(t\geq 0).
\end{equation} 

To establish the stochastic differential equation satisfied by diffusion $X$, we need to express martingale $M$ in terms of a Brownian motion. To do so, we use $J^{\alpha,\beta}_1$ as a drift term and define $\sigma^2(x)=2x(1-x)$, for all $x\in [0,1]$. Then, consider
the stochastic differential equation:
\begin{equation}\label{jacobi-1}
    X_t=X_0+\int_0^tJ_1^{\alpha,\beta}(X_s)ds+\int_0^t\sigma (X_s)dW_s,
\end{equation}
where $W$ is a Brownian motion. A straightforward application of It\^o's formula shows that for any twice-differentiable function $f:T^d\to\R{}$, it holds that
\begin{equation}
    f(X_t)=f(X_0)+\int_0^tJ_1^{\alpha,\beta}(X_s)f'(X_s)ds+\int_0^t\sigma (X_s)f'(X_s)dW_s+\frac{1}{2}\int_0^t\sigma^2(X_s)f''(X_s)ds
\end{equation}
and so
\[f(X_t)-f(X_0)-\int_0^tL_\gamma f(X_s)ds\]
is the martingale
\[\int_0^t\sigma (X_s)f'(X_s)dW_s.\]

In particular, for $f(x)=x$ yields that
\[X_t-\int_0^tJ_1^{\alpha,\beta}(X_s)ds=\int_0^t\sigma (X_s)dW_s.\]

Therefore, $\int_0^t \sigma^2(X_s)ds$ corresponds to the energy dissipated until time $t$,
\begin{equation}\label{efitness}
    \langle M\rangle_t=\int_0^t \sigma^2(X_s)ds.
\end{equation}

In summary, fitness is given by the couple $(X^p,\langle M\rangle)$ introduced earlier in \eqref{ifitness} and \eqref{efitness}. In addition, there is a unique invariant measure , which is indeed an equilibrium measure of the Markov semigroup. It is a Beta distribution with parameters $\alpha,\beta$ ($Beta(\alpha,\beta)$).

\section{The continuous open system dynamics of absolute abundances}\label{sec:ab-ab}

Now we deal with the  biomass flow $X=\seque{X}{t}{[0,\infty[}$ of the absolute abundance of a given biological entity undergoing a birth death process (i.e. genes, cells,  entire organisms, populations or some  trophic level in a food web), so that $X_t\in [0,\infty[$ for all $t\geq 0$. We will show two processes whose fitness structures fit the neutrality considerations, where the equilibrium measure is of Gibbs type. 

Consider the diffusion process \eqref{limit-diff} obtained as a limit in distribution, taking values in $[0,\infty[$, and consider a Hilbert space $L^2([0,\infty[,\vN{B}([0,\infty[),\mu)$, where $\vN{B}([0,\infty[)$ denotes the Borel $\sigma$-field and $\mu(dx)$ is a probability measure absolutely continuous with respect to the Lebesgue measure. As we have seen in the previous case, orthogonal polynomials appear as a natural orthonormal system in the $L^2([0,1])$. However, constructing a polynomial orthonormal basis on $L^2([0,\infty[,\vN{B}([0,\infty[),\mu)$ requires imposing conditions on the measure $\mu$. That problem was studied by Mazet \cite{Mazet}, who proved that a necessary and sufficient condition for the measure $\mu$ to have a polynomial orthonormal basis in $L^2([0,\infty[,\vN{B}([0,\infty[),\mu)$: is that there exists a number $\epsilon>0$ such that
\begin{equation}
    \int_0^\infty e^{\epsilon \abs{x}}\mu(dx)<\infty.
\end{equation}

So, the generator of the diffusion semigroup, in the homogeneous case, is given by the following operator $L$:
\begin{equation}\label{Generator-2}
    Lg(x)=\frac{1}{2}v(x)\frac{\partial^2}{\partial x^2}g(x)+f(x)\frac{\partial}{\partial x}g(x),
\end{equation}
for all $g\in C_0^2([0,\infty[)$.

The corresponding probability current is given by
\begin{equation}\label{current-prob-2}
    J(\rho,x)=f(x)\rho (x)-\frac{1}{2}\frac{\partial}{\partial x}\left( v(x)\rho (x)\right)=\frac{\rho(x)}{2}\left(2f(x)-\frac{\partial}{\partial x}v(x)-v(x)\frac{\partial}{\partial x}\log \rho (x)\right),
\end{equation}
and any density $\rho$ such that $J(\rho,x)=Constant$ for all $x\in [0,\infty[$ is stationary for the above dynamics. In particular, the equilibrium measure $\mu$ is obtained with a density $\rho$ such that
\begin{equation}\label{eq:current-eq}
 2f(x)-\frac{\partial}{\partial x}v(x)-v(x)\frac{\partial}{\partial x}\log \rho (x)=0.   
\end{equation}
\subsection{Fitness potential and scale function for a Gibbs density}

So, to have $\rho\in\vN{G}$, $U=\log \rho-\log Z$ is obtained from the equation
\begin{equation}\label{Potential}
    U(x)=-\log Z-\log v(x)+\int_{0+}^x\frac{2f(y)}{v(y)}dy.
\end{equation}

$\rho$ solves the equation $L_*\rho=0$. Looking for a positive harmonic function $S$ of the diffusion semigroup, one needs to solve the equation 
\begin{equation}
    LS(x)=\frac{1}{2}v(x)\frac{\partial^2}{\partial x^2}S(x)+f(x)\frac{\partial}{\partial x}S(x)=0
\end{equation}
That is, $S$ solves the linear ordinary differential equation
\begin{equation}\label{ODE}
  \frac{\partial^2}{\partial x^2}S(x)+\frac{2f(x)}{v(x)}\frac{\partial}{\partial x}S(x)=0,  
\end{equation}
so, given an initial point $x_0>0$ and a constant $C>0$, and assuming $v(x)>0$ is continuous, one obtains a harmonic function $S$, called the \textit{scale function}, given by
\begin{equation}
   S(x)=C\int_{x_0}^x\exp\left(-\int_0^z\frac{2f (y)}{v(y)}dy\right)dz. 
\end{equation}
Therefore, choosing the constant $C=1$, we have
\[\int_{0+}^x\frac{2f(y)}{v(y)}dy=-\log (\frac{dS}{dx}),\]
so that the fitness potential is related to the scale function by the relation
\begin{equation}
    U(x)=-\log Z-\log v(x)-\log (\frac{dS}{dx}).
\end{equation}
That is, the equilibrium density in terms of the scale function is given by
\begin{equation}
    \rho (x)=\frac{1}{Z}\left(v(x)\frac{dS}{dx}(x)\right)^{-1}.
\end{equation}
Here we assumed that our diffusion takes values in the open interval $]0,\infty[$. Given an arbitrary interval $]a,b[\subset ]0,\infty[$, let us introduce two hitting times $T_a,T_b$, where
\[T_a(\omega)=\inf\set{t\geq 0:X_t(\omega)=a},\;T_b=\inf\set{t\geq 0:X_t(\omega)=b}.\]
Then it has been shown \citep{Klebaner2006} Theorem 6.17 that
\begin{equation}
    \mathbb{P}_x(T_b<T_a)=\frac{S(x)-S(a)}{S(b)-S(a)}=\frac{\int_a^x\exp\left(-\int_{0+}^z\frac{2f(y)}{v(y)}dy\right)dz}{\int_a^b\exp\left(-\int_{0+}^z\frac{2f(y)}{v(y)}dy\right)dz}.
\end{equation}
\subsection{The relation with the Thermodynamics formalism}

It is worth relating our version of the Gibbs equilibrium measure within the thermodynamic formalism.

The equilibrium density can be expressed in thermodynamic Gibbs form by introducing
a free-energy functional. Let
\[
\rho(x)=\frac{1}{Z(\beta)}\exp\{-\beta \mathcal F(x)\},
\]
where \(\mathcal F\) denotes the free energy, \(\beta>0\) is the inverse
temperature, and
\[
Z(\beta)=\int_{\mathcal D}\exp\{-\beta \mathcal F(x)\}\,dx
\]
is the partition function.

Equivalently,
\[
\log \rho(x)=-\beta \mathcal F(x)-\log Z(\beta).
\]
Thus, in the notation \(U=\log\rho\), one has
\[
U(x)= -\beta \mathcal F(x)-\log Z.
\]

Therefore, the free energy is given, up to an additive constant, by
\[
\mathcal F(x)
=
-\frac{1}{\beta}
\int_0^x
\frac{2f(y)-v'(y)}{v(y)}\,dy\]
\[=\frac{1}{\beta}\left(\log \left(\frac{v(x)}{v(0)}\right)-\int_{0+}^x\frac{2f(y)}{v(y)}dy\right)
\]
\[=\frac{1}{\beta}\left(\log \left(\frac{v(x)}{v(0)}\frac{dS}{dx}(x)\right)\right)\]
\vskip 10pt

\[
\mu_{\mathrm{eq}}(dx)
=
\frac{1}{Z}
\exp\!\left\{
\int_0^x
\frac{2f(y)-v'(y)}{v(y)}\,dy
\right\}
dx
\]

\[
\mu_{\mathrm{eq}}(dx)
=
\frac{1}{Z}
e^{-\beta \mathcal F(x)}\,dx.
\]

Here below, we analyze two examples of symmetric diffusions on the positive real line: the Cox-Ingersoll-Ross process and the logistic model.

\subsection{Cox–Ingersoll–Ross (CIR) process}\label{sec:CIR}

When fitness is modeled as an affine function of the, now, absolute biomass $X$, $f(X_t)=a-bX_t$, and volatility is linear $\sigma^2(X_t)=\sigma^2 X_t$ (which regards  ``demographic fluctuations'', see \citep{azaele2006dynamical}), the resulting dynamics is a particular version of a Pearson diffusion \citep{forman2008pearson}, known as the Cox-Ingersoll-Ross (CIR) process, whose stochastic differential equation follows:
\begin{equation}\label{eq:CIR}
 X_t=X_0+\int_0^t(a-bX_s)ds+\sigma\int_0^t\sqrt{X_s}dW_s, 
\end{equation}
where $X_0\geq0$, and $a,b>0$.
Although this process is famous within the field of finance since \citep{cox1985theory}, it has also been considered a model for population growth \citep{feller1951diffusion} and as a neuronal model \citep{sacerdote2012stochastic}.

Whenever  $2a/\sigma^2\geq1$, the process stays positive and has a transition density that follows a non-central chi-square distribution, which has the analytical expression:

\begin{equation}\label{eq:cir-trans}
p(x_t|X_0=x)=w_te^{-u_t+v_t}\left(\frac{u_t}{v_t}\right)^{q/2}I_q(2\sqrt{u_tv_t}),
\end{equation}
for $t>0$, where $w_t=2b/\left(\sigma^2\left[1-e^{-bt}\right]\right)$, $u_t=w_txe^{-bt}$, $v_t=w_tx_t$, $q=2a/\sigma^2-1$ and $$I_q(\xi)=\sum_{k=0}^{\infty}\frac{1}{k!\Gamma(q+k+1)}\left(\frac{\xi}{2}\right)^{q+2k},$$ is the modified Bessel function of the first kind of order $q$.
The conditional mean and variance are given by:
\begin{equation}\label{eq:cir-mean-cond}
\mathbb{E}(X_t|X_0)=\frac{a}{b}+\left(X_0-\frac{a}{b}\right)e^{-bt}
\end{equation}
and
\begin{equation}\label{eq:cir-var-cond}
\mathbb{V}(X_t|X_0)=\frac{a\sigma^2}{2b^2}\left(1-e^{-bt}\right)^2+X_0\frac{\sigma^2}{b}\left(1-e^{-bt}\right)e^{-bt},
\end{equation}
respectively (see \citep{dyrting2004evaluating}). Also, it has a stationary invariant distribution given by:
\begin{equation}\label{eq:gamma-density}
\rho(dx)=\frac{\lambda^{\alpha}}{\Gamma(\alpha)}x^{\alpha-1}e^{-\lambda x}dx,
\end{equation}
with shape $\alpha=2a/\sigma^2$ and rate $\lambda=2b/\sigma^2$  (a $Gamma(2a/\sigma^2,2b/\sigma^2)$; see \citep{ditlevsen2012introduction}, \citep{leonenko2012high}), which has provided good fits for some abundance data (see, e.g., \citep{dennis1984gamma,schmidt1985species, engen1996population, diserud2000general,lande2003stochastic,azaele2006dynamical}). Notice that $\rho(\cdot)$ satisfying $\mu(dx)=\rho(x)dx$ in \eqref{eq:gamma-density} solves \eqref{eq:current-eq} with $f(x)=a-bx$ and $v(x)=\sigma^2x$, and thus, it is a equilibrium density.

When considering $K$ species, we write:
\begin{equation}\label{eq:CIR-several}
dX_i(t)=(a_i-bX_i(t))dt+\sigma\sqrt{X_i(t)}dW_i(t),
\end{equation}
for $i=1, ..., K$, where we assume that the standard Brownian motions, $W_1, ..., W_K$, are independent. 
Note that although the parameter $a$ (which is a parameter that relates the long-term mean growth of biomass) can vary between species through \textit{immigration}, the components of the two-variable fitness are functionally equivalent. Moreover, the independence between Brownian motions sets a long-term  independence between species; therefore, the abundance
variation between them is characteristic of the open system dynamics, which also contains everything that is not specified in the model  (e.g., dispersal, disturbance, and so on). Thus, we can say that this abundance dynamics is given from a neutral perspective (\citep{hubbell2005neutral}, \citep{hubbell2011unified}). 

Also, if we assume that at the time of observing this dynamic, the process is already in a steady state, which means that $X_i(0)\overset{ind}{\sim}Gamma(2a_i/\sigma^2,2b/\sigma^2)$, $i=1, ..., K$ (i.e., the $K$ species start independently with their respective stationary distributions),  by stationarity, for any $t\geq0$, $X_i(t)\overset{ind}{\sim}Gamma(2a_i/\sigma^2,2b/\sigma^2)$, $i=1, ..., K$ and   $X_{i}(t)/S(t)$  follows a $Beta(2a_i/\sigma^2,\sum_{j\neq i}^{K}2a_j/\sigma^2)$ and is independent of $S(t)$, where $S(t)=\sum_{i=1}^KX_i(t)$ (\citep{lukacs1955characterization}).  In this case, we have a result that goes even further: in terms of trajectories, this kind  of transformation will lead to  Jacobi processes for the relative abundances, after a time-change. This result was proved in \citep{gourieroux2006multivariate}, which will be shown in the following, using $K=2$.

Consider $Z_1=f_1(X_1,X_2)=X_1/[X_1+X_2]$. Then 
\[
\frac{\partial f_1}{\partial x_1}(X_1,X_2)=\frac{1-Z_1}{X_1+X_2}, \quad \frac{\partial f_1}{\partial x_2}(X_1,X_2)=\frac{-Z_1}{X_1+X_2};
\]
\[
\frac{\partial^2 f_1}{\partial x_1^2}(X_1,X_2)=\frac{-2(1-Z_1)}{(X_1+X_2)^2}\text{,  and  } \frac{\partial^2 f_1}{\partial x_2^2}(X_1,X_2)=\frac{2Z_1}{(X_1+X_2)^2},
\]
and applying Itô's formula we obtain:
\[
dZ_1(t)=\sum_{i=1}^{2}\frac{\partial f_1}{\partial x_i}(X_1(t),X_2(t))dX_i(t)+\frac{\sigma^2}{2}\sum_{i=1}^{2}\frac{\partial^2 f_1}{\partial x_i^2}(X_1(t),X_2(t))X_i(t)dt
\]
\[
=\frac{-\tilde{a}}{X_1(t)+X_2(t)}(Z_1(t)-\tilde{b})dt+\frac{\sqrt{\sigma^2Z_1(t)(1-Z_1(t))}}{\sqrt{X_1(t)+X_2(t)}}\left(\sqrt{1-Z_1(t)}dW_1(t)-\sqrt{Z_1(t)}dW_2(t)\right),
\]
with $\tilde{a}=a_1+a_2$ and $\tilde{b}=a_1/(a_1+a_2)$. Recall Itô's isometry: for a process $H=(H_t)_{t\in\mathbb{R}_+}$ measurable and square-integrable and a  Brownian motion $W$, for any $t\geq0$ we have $\mathbb{E}\left(\int_{0}^tH_sdW_s\right)^2=\mathbb{E}\left(\int_{0}^tH_s^2ds\right)$. Now, notice that the term \[\tilde{W}(t):=\int_{0}^{t}\left(\sqrt{1-Z_1(s)}dW_1(s)-\sqrt{Z_1(s)}dW_2(s)\right)\] actually defines a Brownian motion: it also defines a local martingale (with respect to adequate filtration)  with $\mathbb{E}\left(\tilde{W}(t)\right)=0$ and  \[\mathbb{E}\left(\tilde{W}(t)^2\right)=\mathbb{E}\left(\int_{0}^t\sqrt{1-Z_1(s)}dW_1(s)\right)^2+\mathbb{E}\left(\int_{0}^t\sqrt{Z_1(s)}dW_2(s)\right)^2=t,\] by Itô's isometry (formally, it follows by Lévy's \textit{characterization of Brownian Motion} since $\langle \tilde{W}\rangle_t=t$, and then $\left(\tilde{W}(t)^2-t\right)_{t\in\mathbb{R}_+}$ is a continuous local martingale, with respect to adequate filtration). Hence, the SDE for $Z_1$ reads:
\[
dZ_1(t)=\frac{-\tilde{a}(Z_1(t)-\tilde{b})}{S(t)}dt+\frac{\sqrt{\sigma^2Z_1(t)(1-Z_1(t))}}{\sqrt{S(t)}}d\tilde{W}(t),
\]
where $S=X_1+X_2$. Notice that  $S$  follows a CIR process:
\[
dS(t)=(\tilde{a}-bS(t))dt+\sigma\sqrt{S(t)}\left(\sqrt{Z_1(t)}dW_1(t)+\sqrt{Z_2(t)}dW_2(t)\right),
\]
where the term $W^*(t):=\int_{0}^{t}\left(\sqrt{Z_1(s)}dW_1(s)+\sqrt{Z_2(s)}dW_2(s)\right)$ defines again a  Brownian motion by an analogous argument to the above. 

Finally, by applying the time-change and the new variable
\[
\tau=\phi(t)=\int_{0}^{t}\frac{1}{S(s)}ds,\qquad\left(\mathcal{Z}_1(\tau)=Z_1(\phi^{-1}(\tau))\right)_{\tau\in\mathbb{R}_+},
\]
we find that the process $\mathcal{Z}_1$ satisfies:
\begin{align*}
\mathcal{Z}_1(\tau) &= Z_1(t)
\\
&=Z_1(0)+\int_{0}^{t}\frac{-\tilde{a}(Z_1(s)-\tilde{b})}{S(s)}ds+\sigma\int_{0}^{t}\frac{\sqrt{Z_1(s)(1-Z_1(s))}}{\sqrt{S(s)}}d\tilde{W}(s)
\\
&=\mathcal Z_1(0)+\int_{0}^{\phi^{-1}(\tau)}\frac{-\tilde{a}(Z_1(s)-\tilde{b})}{S(s)}ds+\sigma\int_{0}^{\phi^{-1}(\tau)}\frac{\sqrt{Z_1(s)(1-Z_1(s))}}{\sqrt{S(s)}}d\tilde{W}(s).
\end{align*}
In the first integral, we use the change of variables $r=\phi(s)$ to find
\[
\int_{0}^{\phi^{-1}(\tau)}\frac{-\tilde{a}(Z_1(s)-\tilde{b})}{S(s)}ds = \int_{0}^{\tau}-\tilde{a}(Z_1(\phi^{-1}(r))-\tilde{b})dr = \int_{0}^{\tau}-\tilde{a}(\mathcal Z_1(r)-\tilde{b})dr,
\]
and in the second integral
\begin{align*}
\int_{0}^{\phi^{-1}(\tau)}\frac{\sqrt{Z_1(s)(1-Z_1(s))}}{\sqrt{S(s)}}d\tilde{W}(s)
&= \int_0^\tau \sqrt{Z_1(\phi^{-1}(r))(1-Z_1(\phi^{-1}(r)))}d\tilde{W}(r)
\\
&= \int_0^\tau \sqrt{\mathcal Z_1(r)(1-\mathcal Z_1(r))}d\tilde{W}(r),
\end{align*}
summarizing
\begin{align*}
\mathcal Z_1(\tau)&=\mathcal{Z}_1(0)-\int_{0}^{\tau}\tilde{a}(\mathcal{Z}_1(r)-\tilde{b})dr+\sigma\int_{0}^{\tau}\sqrt{\mathcal{Z}_1(r)(1-\mathcal{Z}_1(r))}d\tilde{W}(r),
\end{align*}
which is an equation following \eqref{jacobi-1} with $\tilde{a}=\alpha+\beta+2$, $\tilde{b}=(\beta+1)/(\alpha+\beta+2)$ and $\sigma^2=2$.

This result can be generalized for more species ($K\geq2$), where the resulting relative abundance processes can be arranged to be a multivariate Jacobi process, whose joint invariant distribution follows a Dirichlet distribution, as established in Theorem \ref{th:dirch-inv}, for $d=K$ (see \citep{gourieroux2006multivariate}).

\begin{rem}
Actually, the above result can be set under weaker conditions. It is enough for $X_i's$ to start with positive biomass and $2a_i/\sigma^2\geq1$. This will guaranty that the whole community $S(\cdot)$ stays strictly positive with probability one (specifically, it will follow a strictly positive CIR; see, e.g., \citep{mao2007stochastic}, p. 308), and therefore, the integral in the time-change $\tau=\phi(t)=\int_0^tS(s)^{-1}ds$ will be  well defined. Moreover, by the ergodic property of $S(\cdot)$, we have $(1/t)\int_0^tS(s)^{-1}ds\rightarrow\int_0^{\infty}x^{-1}\rho_{S}(x)dx$, almost surely, when $t\rightarrow\infty$, where $\rho_{S}(\cdot)$ is the corresponding ergodic invariant Gamma density of the process $S(\cdot)$, implying that $\phi(t)\rightarrow\infty$, almost surely, when $t\rightarrow\infty$, which must be considered a suitable time.
\end{rem} 

In Section \ref{sec:stat-cir} we show an estimation procedure based on the trajectories instead of using the probability transitions \eqref{eq:cir-trans}, since the latter has an analytical expression and, under this, we have to approximate it. Given that the state space is unbounded, we can use a convenient transformation of the original process, and then build a Radon-Nikodym derivative to further derive the parameter estimators via the maximum likelihood ratio.

\subsection{The stochastic logistic process}\label{sec:logistic}

We can consider another ecological derivation under the assumption of neutrality, which also assumes that the species  are functionally equivalent, and the abundance variation between them is mainly due to variability in the noise. Consider the biomass abundance of $K$ species following a so-called \textit{stochastic logistic dynamics} (\citep{pasquali2001stochastic}), given by:
\begin{equation}\label{eq:st-logistic}
dY_i(t)=[r_i-\eta Y_i(t)]Y_i(t)dt+\sigma Y_i(t)dW_i(t),    
\end{equation}
for each $i=1,...,K$, with $Y_i(0)>0$, where $r_i>0$ is the intrinsic mean growth rate of species $i$ and $\eta>0$ is the common species-specific strength of density dependence.  The $W_i$'s are mutually independent Brownian motions and $\sigma>0$ is the species-common diffusion coefficient. These can be interpreted as environmental fluctuations around the $r_i$'s (\citep{may2001stability}), as we will see below. Anyway,  such an equation \eqref{eq:st-logistic} has an explicit solution, given by:
\begin{equation}\label{eq:exac-sol-SL}
Y_i(t)=\exp\left(\left[r_i-\frac{\sigma^2}{2}\right]t+\sigma W_i(t)\right)\left[\frac{1}{Y_i(0)}+\eta\int_{0}^{t}\exp\left(\left[r_i-\frac{\sigma^2}{2}\right]s+\sigma W_i(s)\right)ds\right]^{-1}
\end{equation}
(see \citep{arato2003famous}).

Such common parameters and functional equivalent growths regard to a fitness equivalence between species, which are ``interacting'' independently, thus producing a neutral model of species abundance dynamics. Additionally, these processes have a stationary distribution given by a $Gamma(u_i,v)$, whose probability density is
\begin{equation}\label{eq:gamma-invariant}
\rho_i(dy)=\frac{y^{u_i-1}e^{-vy}v^{u_i}}{\Gamma(u_i)}dy,
\end{equation}
with shape $u_i=2r_i/\sigma^2-1$ and rate $v=2\eta/\sigma^2$, whenever  $2r_i/\sigma^2>1$, for all $i=1,...,K$. As mentioned previously in \ref{sec:CIR}, the Gamma distribution as a model for species abundance has been documented as appropriate for some cases (see, e.g., \citep{dennis1984gamma}, \citep{schmidt1985species}, \citep{engen1996population}, \citep{diserud2000general}, \citep{azaele2006dynamical}).
Notice they share the stationary density of the CIR model \eqref{eq:CIR-several}, taking $a_i=r_i-\sigma^2/2$ and $b=\eta$
(see \citep{azaele2006dynamical}). Notice also that for each $i$, $\rho_i(\cdot)$ satisfying $\mu_i(dx)=\rho_i(x)dx$ in \eqref{eq:gamma-invariant} solves \eqref{eq:current-eq} with $f(x)=(r_i-\eta x)x$ and $v(x)=\sigma^2 x^2$, and thus it is an equilibrium density.
However, the stochastic logistic process \eqref{eq:st-logistic}  contains biologically relevant quantities in which fitness volatility appears as an explicit diffusion of species fitness. Specifically, let $R_i(t)$ be the intrinsic growth rate of species $i$ at time $t$, and assume that it is a semimartingale satisfying:
\[
dR_i(t)=r_idt+\sigma dW_i(t).
\]
Then \eqref{eq:st-logistic} is equivalent to the \textit{semimartingale-driven} stochastic differential equation:
\[
dY_i(t)=-\eta Y_i(t)^2dt+Y_i(t)dR_i(t).
\]

Now, if we assume that all the processes $Y_i$'s start from independent $Gamma(u_i,v)$'s (i.e., we consider stationarity at the beginning, as we set in the previous case), then, for each $i=1,...,K$ and $t\geq0$, the relative abundances are as follows.
\begin{equation}\label{eq:relative-abundance}
Z_i(t)=\frac{Y_i(t)}{\sum_{j=1}^{K}Y_j(t)}\sim Beta(u_i,\sum_{j=1}^{K}u_j-u_i).
\end{equation}
In general, for any $t\in\mathbb{R}_+$, $\mathbf{Z}(t)=(Z_1(t), ...., Z_k(t))^\top$ follows the invariant law given in Th. \ref{th:dirch-inv} with $d=K$.

In Section \ref{sec:stat-sl}, and under the same arguments given for the CIR case, we also show an estimation procedure based on the trajectories, using a convenient transformation first of the original process, and then,  building  a Radon-Nikodym derivative, to further derive the parameter estimators via maximum likelihood ratio as well. 

\section{Statistical methods}\label{sec:stat-meth}

\subsection{The Jacobi process}\label{sec:stat-jac}

We will start by proceeding with the statistical analysis under the Jacobi process for $d=1$. For such purposes, it is convenient to rewrite Equation \eqref{jacobi-1} with the following reparameterization:
\begin{equation}\label{eq:jacobi-repar}
dX_t=-a(X_t-b)dt+\sqrt{2X_t(1-X_t)}dW_t,
\end{equation}
where $a=\alpha+\beta+2>0$ and $b=(\beta+1)/(\alpha+\beta+2)\in(0,1)$. This equation has an invariant stationary distribution,  given by the Beta distribution with parameters $ab=\beta+1$ and $a(1-b)=\alpha+1$, $Beta(\beta+1,\alpha+1)$ (see \citep{gourieroux2006multivariate}).
A quasi-likelihood method can be applied to estimate the parameter $\theta=(\alpha,\beta)$ based on the spectral decomposition of the  generator: for twice differentiable functions $f$ defined on the one-dimensional simplex, it will be given by:
\begin{equation}\label{generator-repar}
\mathcal{L}f(x)=-a(x-b)\frac{\partial f}{\partial x}(x)+x(1-x)\frac{\partial^2 f}{\partial x^2}(x).    
\end{equation}
This generator admits a spectral decomposition with a set of eigenvalues $\{\lambda_n\}_{n\in\mathbb{N}}$ and eigenvectors $\{P_n\}_{n\in\mathbb{N}}$, such that:
\begin{equation}\label{eq:spectral}
\mathcal{L}P_n=\lambda_nP_n,
\end{equation}
for all $n\in\mathbb{N}$. The eigenvalues are given by $\lambda_n=-an-n(n-1)=-(\alpha+\beta+2)n-n(n-1)$ and the eigenvectors by:
\[
P_n(x)=\left(\frac{\Gamma(\beta+n+1)(2n+\alpha+\beta+1)\Gamma(\alpha+1)\Gamma(\beta+1)}{n!\Gamma(\alpha+\beta+n+1)\Gamma(\alpha+\beta+2)\Gamma(\alpha+n+1)}\right)^{1/2}
\]
\begin{equation}\label{eigenfunctions}
\times\sum_{m=0}^{n}(-1)^{m}\binom{n}{m}\frac{\Gamma(\alpha+\beta+n+m+1)}{\Gamma(\beta+m+1)}x^m,
\end{equation}
which are the Jacobi polynomials (see \cite{wong1964construction}). This is a case of the spectral decomposition given in  Theorem \ref{Th_Dunkl-Xu} under $d=1$, $\gamma_1=\beta$ and $\gamma_2=\alpha$.

Now, in order to describe an estimation method for $\theta$, we can rely on some procedures  presented in the works \citep{gourieroux2004estimation} and \citep{valery2011quasi} (see also \citep{kessler1999estimating}, \citep{larsen2007diffusion}). Specifically, we will describe a quasi-likelihood method based on the spectral decomposition \eqref{eq:spectral}. Here, we consider that we  observe the process \eqref{eq:jacobi-repar} in discrete times, at equal lags, which can be assumed as 1. The transition density, $\{p(x_{t}|x_{t-1};\theta)\}_t$, is needed to compute such a quasi-likelihood method. However, this transition density has no closed form expression, for which we need an approximate form of this (though, an alternative to the maximum likelihood estimation approach for intractable transition densities for discretely observed diffusions can be found in
\citep{ait2002maximum}; but here we will only present the case to be described, which refers to the specific case of the Jacobi process). For inferential purposes, we will write the polynomials \eqref{eigenfunctions} as $P_n(x)=P_n(x;\theta)$, and the eigenvalues as $\lambda_n=\lambda_n(\theta)$.

An approximation of the above transition density via the spectral decomposition is given by:
\begin{equation}\label{spectral-app}
p_N(x_{t}|x_{t-1};\theta)=p(x_{t};\theta)\left(1+\sum_{n=1}^{N}\exp(\lambda_n(\theta))P_n(x_{t};\theta)P_n(x_{t-1};\theta)\right),
\end{equation}
where $p(x;\theta)$ is the $Beta(\beta+1,\alpha+1)$ density, and $$\lim_{N\rightarrow\infty}p_N(x_{t}|x_{t-1};\theta)=p(x_{t}|x_{t-1};\theta).$$ 
So, the maximum likelihood estimator of $\theta$, given by $$\hat{\theta}_{T}^{ML}=\text{arg max}_{\theta}\sum_{t=1}^{T}\log(p(x_{t}|x_{t-1};\theta)),$$ will be replaced by considering the quasi-likelihood function, $$\hat{\theta}_{T,N}^{MQL}=\text{arg max}_{\theta}\sum_{t=1}^{T}\log(p_N(x_{t}|x_{t-1};\theta)).$$

\begin{rem}
Notice that, for this method, it is assumed that we are starting from its stationary distribution.
\end{rem}

There is no closed-form solution for the normal equations of $$\sum_{t=1}^{T}\log(p_N(x_{t}|x_{t-1};\theta))$$ to obtain explicitly $\hat{\theta}_{T,N}^{MQL}$. Thus, numerical methods have to be implemented. In such methods, initial values (or seeds) are needed, which can be chosen by some ``pre-estimation'' of the parameter involved. As the Jacobi process is ergodic and compact supported, trajectorial averages converges to the corresponding invariant averages for ``large times''; therefore, some preliminary trajectorial average estimates can be used to set the referred seeds. In fact, these ``pre-estimates'' are already close to the true values (see Figure \ref{fig:Jacobi}).

Specifically, let $X_{[0,T]}$ be the solution of the process \eqref{eq:jacobi-repar} up to a time horizon $T$. Let $u_T=\text{mean}(X_{[0,T]})$ and $v_T^2=\text{var}(X_{[0,T]})$ the trajectorial average and variance of $X_{[0,T]}$, respectively.
Given that the stationary distribution of  \eqref{eq:jacobi-repar} follows a $Beta(s_1,s_2)$, with $s_1=\beta+1$ and $s_2=\alpha+1$, and it is ergodic (in particular, it implies that the trajectorial moments converge to the moments w.r.t. the stationary distribution, when $T\rightarrow\infty$),  we have that the trajectorial moment estimators for $\theta=(\alpha,\beta)$ are obtained by solving  the equations:
\begin{equation}\label{moment-traj}
s_1 =\left(\frac{1-u_T}{v_T^2} - \frac{1}{u_T}\right)u_T^2,\quad s_2 =s_1\left(\frac{1}{u_T}-1\right),
\end{equation}
for a ``large'' $T>0$. For the cases simulated in Figure \ref{fig:Jacobi}, the trajectorial moment estimators were $\hat{\theta}=(\hat{\alpha},\hat{\beta})=(10.16, 6.40)$ and $\hat{\theta}=(\hat{\alpha},\hat{\beta})=(1.16, 4.06)$, for the top processes, from left to right, respectively; and $\hat{\theta}=(\hat{\alpha},\hat{\beta})=(9.33, 5.54)$ and $\hat{\theta}=(\hat{\alpha},\hat{\beta})=(1.05, 3.09)$,  for the bottom processes, from left to right, respectively. So, such ``pre-estimates'' can give us good starting points to subsequently applied the quasi-likelihood method described above.

\begin{figure}[h!]
\centering
\resizebox{6cm}{!}{\includegraphics{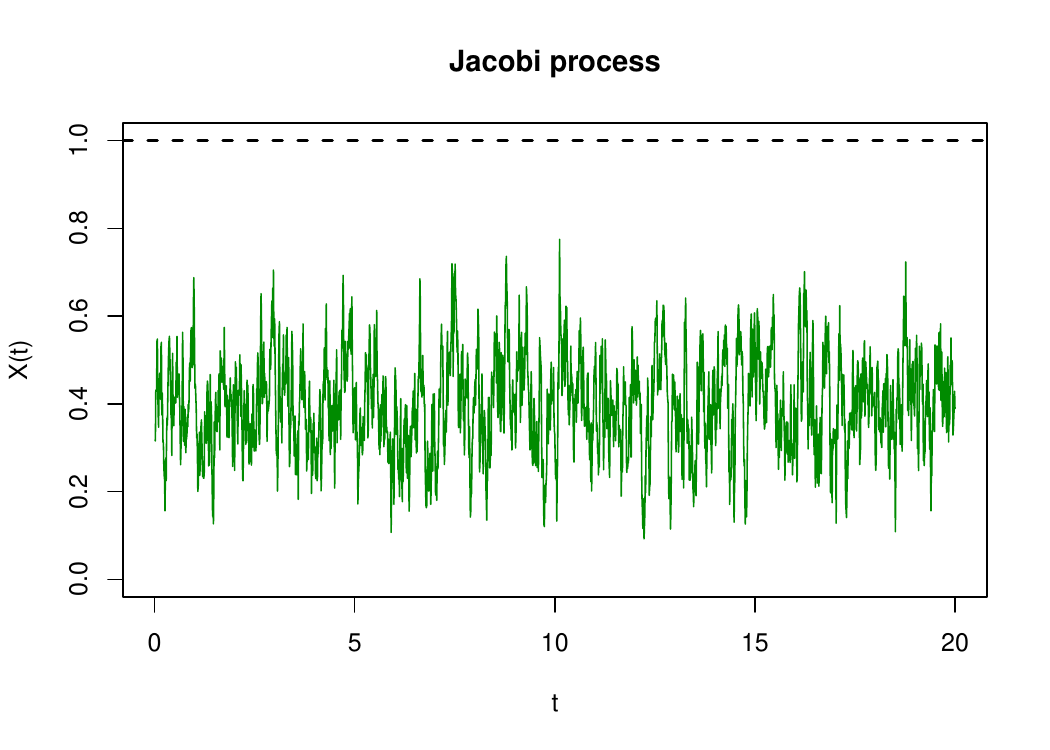}}
\resizebox{6cm}{!}{\includegraphics{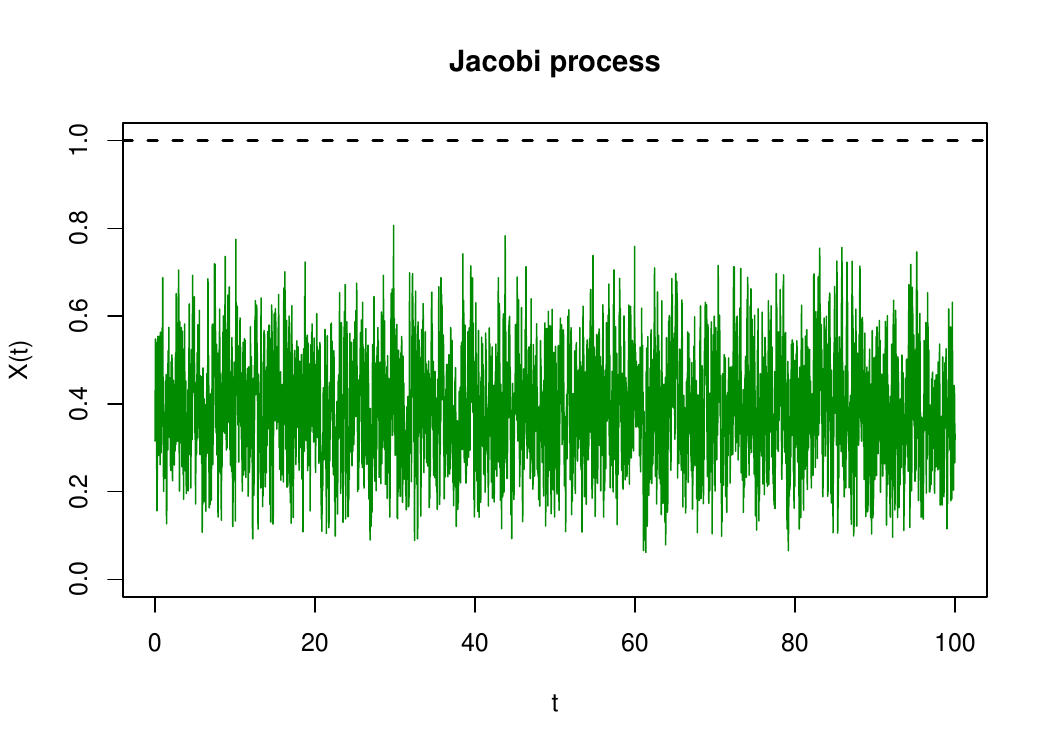}}
\resizebox{6cm}{!}{\includegraphics{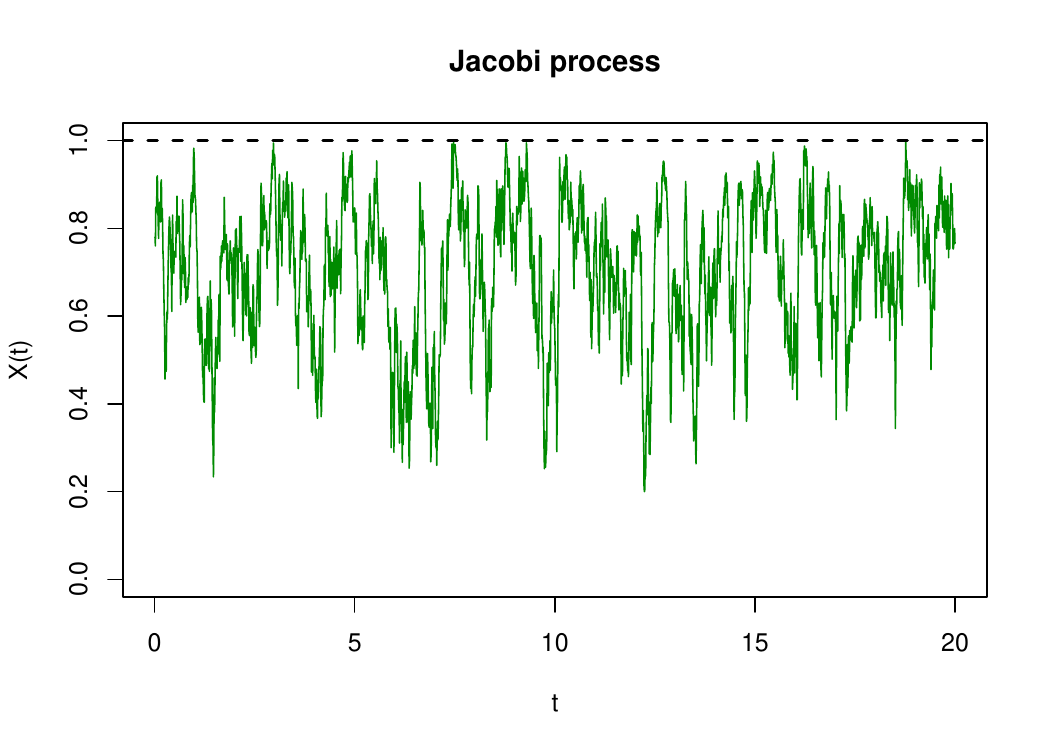}}
\resizebox{6cm}{!}{\includegraphics{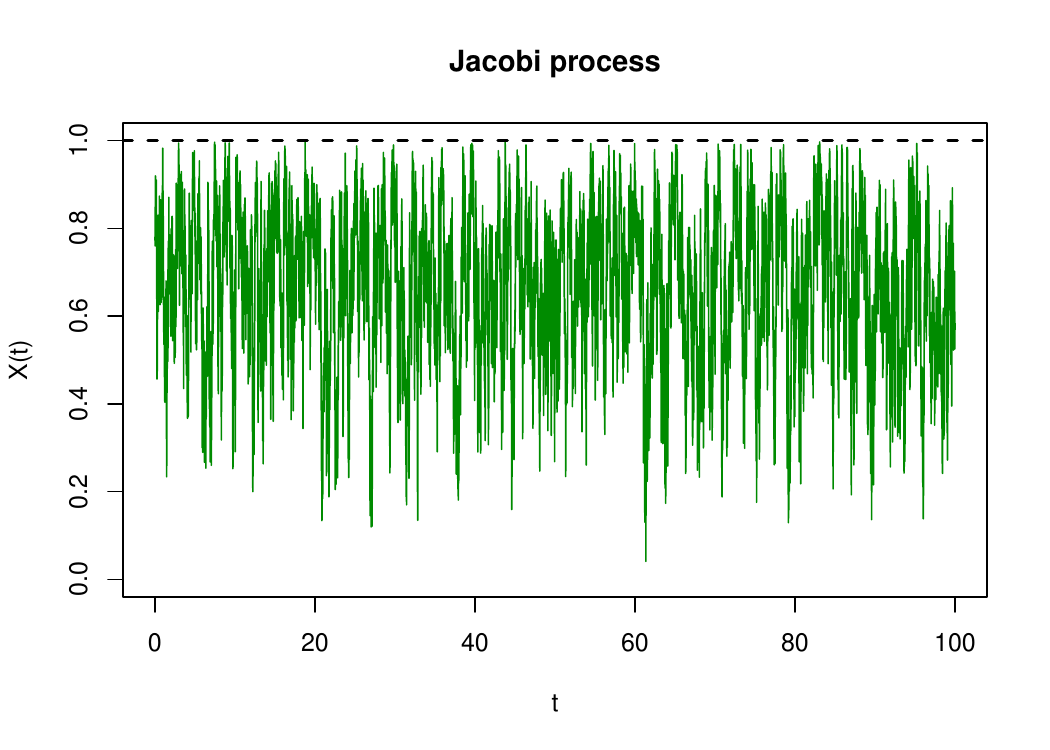}}
\caption{Simulations of Jacobi processes following Equation \eqref{eq:jacobi-repar}, by using Euler-Maruyama method with discretizations $dt\approx \delta=0.01$. Top: (left) time horizon $T=20$ and $\theta=(\alpha,\beta)=(10,6)$; (right) time horizon $T=20$ and $\theta=(\alpha,\beta)=(1,3)$. Bottom: the same as the top figures, but with a time horizon $T=100$. All the initials stating points for the simulated process $X$ were taken as random from its corresponding stationary distribution.}
\label{fig:Jacobi}
\end{figure}

The simulation of the process, together with the implemented quasi-likelihood method, whose initial values were taken from Equations \eqref{moment-traj}, can be found in \textcolor{blue}{https://github.com/RplazaValpo/Estimation-Jacobi-Process}, for \textbf{R} software. Here we used $N=5$ in \eqref{spectral-app}. The numerical optimization algorithm used was the \textbf{L-BFGS-B}, which is a quasi-Newton method that presents an improvement in accuracy, especially when the number of variables to be optimized is high; see \citep{byrd1995limited}.
For our simulated cases, the quasi-likelihood estimates were $\hat{\theta}_{T,N}^{MQL}=(\hat{\alpha}_{T,N}^{MQL},\hat{\beta}_{T,N}^{MQL})=(10.04, 6.32)$ and $\hat{\theta}_{T,N}^{MQL}=(\hat{\alpha}_{T,N}^{MQL},\hat{\beta}_{T,N}^{MQL})=(1.11, 3.98)$, for the top processes, from left to right, respectively; and $\hat{\theta}_{T,N}^{MQL}=(\hat{\alpha}_{T,N}^{MQL},\hat{\beta}_{T,N}^{MQL})=(9.27, 5.50)$ and $\hat{\theta}_{T,N}^{MQL}=(\hat{\alpha}_{T,N}^{MQL},\hat{\beta}_{T,N}^{MQL})=(1.00, 2.98)$,  for the bottom processes, from left to right, respectively. Thus, in almost all the cases considered, these estimates had a little improvement with respect to the corresponding moment estimators.

In the case of Jacobi processes for higher dimensions, it is difficult to derive a tractable approximation of the transition density as in \eqref{generator-repar} (\citep{gourieroux2006multivariate}). Thus, we cannot directly obtain  a quasi-likelihood method as in the previous case. However, given that the process is ergodic, with invariant probability measure following a Dirichlet (see Th. \ref{th:dirch-inv}), some approach can be done with the trajectorial marginals, in which, under stationarity,  each marginal has a Beta distribution with parameters $s_1^{(i)}=\gamma_i+1$ and $s_2^{(i)}=\langle \gamma,\1\rangle+d-\gamma_i$ (see also \citep{ng2011dirichlet}).

\subsection{The CIR process}\label{sec:stat-cir}

As shown in \eqref{eq:cir-trans}, to estimate their parameters through the transition probabilities, we also have to obtain some approach to these. However, given the unboundedness of the state space, we can apply some transformation of the original process to obtain a \textit{likelihood ratio}, and then get the parameter estimates.

Let $X$ be a stationary CIR process following \eqref{eq:CIR}. Let  $Z=f(X)=2X^{1/2}$. Then, by applying Itô's formula, we obtain:
\begin{equation}\label{eq:ito-cir}
dZ(t)=G(Z(t))dt+\sigma dW(t),
\end{equation}
with $G(z)=(1/z)(2a-bz^2/2-\sigma^2/2)$. Thus, the parameter $\sigma^2$ can be exactly obtained by the continuity of the trajectories of \eqref{eq:ito-cir}: for any $t>0$, we have
\[
V^p(Z(t))=\langle Z-Z^p \rangle_t = \sigma^2 t,
\]
where
\[
\sum_{k=1}^{2^n}\left(Z(tk/2^n)-Z(t(k-1)/2^n\right)^2\overset{\mathbb{P}}{\longrightarrow}\langle Z-Z^p \rangle_t,
\]
(convergence in the probability sense) when $n\rightarrow\infty$. Thus, if we observe the process discretely  over a fixed time window $[0,T]$, an estimator of $\sigma^2$ is given by:
\begin{equation}\label{eq:sigma-est}
\widehat{\sigma^2}=\frac{1}{T}\sum_{k=1}^{N}\left(Z(t_k)-Z(t_{k-1})\right)^2,
\end{equation}
where $t_0=0$ and $t_N=T$. So, from now on, we will proceed as if $\sigma^2$ were known, so the parameter to be estimated will be $\vartheta=(a,b)$.

That being said, let  $\mathbb{P}^T_{Z}$ and $\mathbb{P}^T_{B}$ the probability measures induced by  $(Z(t))_{t\in[0,T]}$ and  $(\sigma W(t))_{t\in[0,T]}$, respectively. We see that if we set $Z(0)=z_0>0$ (which is implied by setting $X(0)=x_0>0$), $\mathbb{P}^T_{Z}$ is absolutely continuous with respect to $\mathbb{P}^T_{B}$, and   the Radon-Nikodym derivative of $\mathbb{P}^T_{Z}$ w.r.t $\mathbb{P}^T_{B}$ is given by:
\begin{equation}\label{eq:R-N-cir}
\frac{d\mathbb{P}^T_{Z}}{d\mathbb{P}^T_{B}}=\exp\left(\int_0^T\frac{G(Z(t))}{\sigma^2}dZ(t)-\frac{1}{2}\int_0^T\frac{G(Z(t))^2}{\sigma^2}dt\right)=:L_T(\vartheta; Z).
\end{equation}
We can regard this as a likelihood ratio, and thus the maximum likelihood ratio estimator of $\vartheta=(a,b)$, is obtained by maximizing the right-hand side of \eqref{eq:R-N-cir} (in practice, its logarithm) with respect to $a$ and $b$ (see Theorem 1.12 in \citep{kutoyants2004statistical}). Notice that, for its well-posedness, it must be satisfied:
\begin{equation}\label{eq:cond-RD-CIR}
\mathbb{P}\left(\int_0^T\frac{G(Z(t))^2}{\sigma^2}dt<\infty\right)=1,
\end{equation}
for any $T>0$, which holds under the \textit{Feller condition}: $2a/\sigma^2>1$  (see \citep{mao2007stochastic}, p. 308), implying that $\mathbb{P}(X(t)>0, t\geq 0)=1$, and so $(Z(t)>0, t\geq 0)$ does as well. 

So,  the maximum likelihood-ratio estimator for $\vartheta=(a,b)$ is given by:
\begin{equation}\label{eq:ML-CIR}
(\hat{a}_{ML}, \hat{b}_{ML})=\text{argmax}_{\vartheta\in \mathbb{R}_+\times\mathbb{R}_+} l_T(\vartheta; Z),
\end{equation}
where $l_T(\vartheta; Z)=\log L_T(\vartheta; Z)$.

However, there are some simple approaches whose effectiveness lies in the assumption of ergodicity: given that its stationary distribution of $X(t)$ from \eqref{eq:CIR}  follows a $Gamma(2a/\sigma^2,2b/\sigma^2)$, for any $t\geq0$, we can use its trajectorial moments to obtain the moments equations over a time horizon $[0,T]$:
\[
\frac{a}{b}=\text{mean}(X_{[0,T]})=\frac{1}{T}\int_{0}^{T}X(s)ds, \]
and 
\[\frac{a\sigma^2}{2b^2}=\text{var}(X_{[0,T]})\]
\[=\frac{1}{T}\int_{0}^{T}[X(s)-\text{mean}(X_{[0,T]})]^2ds
\]
\[
=\frac{1}{T}\int_{0}^{T}X(s)^2ds-\left(\frac{1}{T}\int_{0}^{T}X(s)ds\right)^2,
\]
where $a/b$ and $a\sigma^2/(2b^2)$ are the limits of \eqref{eq:cir-mean-cond} and \eqref{eq:cir-var-cond} when $t\rightarrow\infty$, respectively. Solving these equations for $a$ and $b$, we obtain their moment estimates, $\hat{a}_{MM}$ and $\hat{b}_{MM}$. Their are given by:
\begin{equation}\label{eq:a-mm}
\hat{a}_{MM}=\frac{\hat{b}_{MM}}{T}\int_{0}^{T}X(s)ds,
\end{equation}
and 
\begin{equation}\label{eq:b-mm}
\hat{b}_{MM}=\frac{\widehat{\sigma^2}\int_{0}^{T}X(s)ds}{2T\left[\frac{1}{T}\int_{0}^{T}X(s)^2ds-\left(\frac{1}{T}\int_{0}^{T}X(s)ds\right)^2\right]},
\end{equation}
which are simple expressions to approach. By ergodicity (which is guaranteed by the above Feller condition), such estimates converge to the true parameters when $T\rightarrow\infty$.

Despite the simplicity of the above method of moments, the maximum likelihood ratio method is generally preferable, as it tends to be more efficient than the method of moments in this type of ergodic diffusion processes (see \citep{kutoyants2004statistical}).

\subsection{The SL process}\label{sec:stat-sl}

Here, we proceed with the parameter estimation of a single stationary logistic process by relying on methods based on the observable trajectories instead of the transition density approach. First, we can transform our single logistic equation,
\begin{equation}\label{eq:single-logistic}
dY(t)=[r-\eta Y(t)]Y(t)dt+\sigma Y(t)dW(t),
\end{equation}
into a diffusion process on $\mathbb{R}$ by considering $Z=\log(Y)$ instead of $Y$, and applying the Itô formula:
\begin{equation}\label{eq:log-logistic}
dZ(t)=H(Z(t))dt+\sigma dW(t),
\end{equation}
with $H(z)=r-\eta e^z-\sigma^2/2$. Here, the parameter $\sigma^2$ can also be obtained exactly by the continuity of the trajectories, and then we use \eqref{eq:sigma-est} to estimate its position. Given this, our parameter of interest is  $\vartheta=(r,\eta)$.

Now, we proceed in analogous way as in \ref{sec:stat-cir}:  Let  $\mathbb{P}^T_{Z}$ and $\mathbb{P}^T_{W}$ the probability measures induced by  $(Z(t))_{t\in[0,T]}$ and  $(\sigma W(t))_{t\in[0,T]}$, respectively. We find that $\mathbb{P}^T_{Z}$ is absolutely continuous with respect to $\mathbb{P}^T_{W}$, and   the Radon-Nikodym derivative of $\mathbb{P}^T_{Z}$ w.r.t $\mathbb{P}^T_{W}$ is given by:
\begin{equation}\label{eq:R-N-log}
\frac{d\mathbb{P}^T_{Z}}{d\mathbb{P}^T_{W}}
=\exp\left(\int_0^T\frac{H(Z(t))}{\sigma^2}dZ(t)-\frac{1}{2}\int_0^T\frac{H(Z(t))^2}{\sigma^2}dt\right)=:L_T(\vartheta; Z).
\end{equation}

This is a likelihood ratio, and thus the estimator of the maximum likelihood ratio of $\vartheta=(r,\eta)$, $\hat{\vartheta}_{ML}=(\hat{r}_{ML},\hat{\eta}_{ML})$ is obtained by maximizing the right-hand side of \eqref{eq:R-N-log} (in practice, its logarithm) with respect to $r$ and $\eta$ (see Theorem 1.12 in \citep{kutoyants2004statistical}), which is given by:
\begin{equation}\label{eq:ML-ST}
(\hat{r}_{ML}, \hat{\eta}_{ML})=\text{argmax}_{\vartheta\in \mathbb{R}_+\times\mathbb{R}_+} l_T(\vartheta; Z),
\end{equation}
where $l_T(\vartheta; Z)=\log L_T(\vartheta; Z)$.

Alternatively, and analogously to what was done for the CIR and Jacobi processes, based on the stationarity of the process, we can obtain the moment estimators by solving the moment equations:
\begin{equation}\label{eq:mom-eq-gamm}
\frac{u}{v}=\frac{1}{T}\int_{0}^{T}Y(s)ds,\quad \frac{u}{v^2}+\left(\frac{u}{v}\right)^2=\frac{1}{T}\int_{0}^{T}Y(s)^2ds,
\end{equation}
whose solution gives rise:
\begin{equation}\label{eq:u-est-MM}
\hat{u}_{MM}=\hat{v}_{MM}\frac{1}{T}\int_{0}^{T}Y(s)ds,
\end{equation}
and 
\begin{equation}\label{eq:v-est-MM}
\hat{v}_{MM}=\frac{\frac{1}{T}\int_{0}^{T}Y_sds}{\frac{1}{T}\int_{0}^{T}Y(s)^2ds-\left(\frac{1}{T}\int_{0}^{T}Y(s)ds\right)^2},
\end{equation}
where we conclude $\hat{r}_{MM}=\widehat{\sigma^2}(\hat{u}_{MM}+1)/2$ and $\hat{\eta}_{MM}=\widehat{\sigma^2}\hat{v}_{MM}/2$, and which converge to the true parameters as $t\rightarrow\infty$, by ergodicity.

As in the previous case, the maximum likelihood ratio method  will be preferable to the moment method, due to the improvement in efficiency  (see \citep{kutoyants2004statistical}).

\section{Discussion}

\vskip 10pt
This article investigates the meaning of fitness as the predictable motion of living biomass and entities under the light of open system dynamics. To do so, we use the mathematical formalism of the General Theory of Processes, which is a theoretical basis of modern Stochastic Analysis. 

The General Theory of Processes (GTP), originated in the 1970's and 1980's, provides a suitable mathematical representation of history in open system dynamics, giving a consistent classification of predictable processes, stopping times, and a well founded stochastic integration calculus. The GTP representation of an open system dynamics (OSD) is nicely obtained through the notion of semimartingale which admits a decomposition in two parts: a finite-variation process that corresponds to the main system dynamics, and a martingale part (noise) representing the relation of the main system with the environment, with no finite-variation trajectories. This decomposition is not unique, but in physical and biological applications one deals with a rich class of so called \textit{special semimartingales} that have a unique decomposition where the finite-variation part is a predictable process. The noise has no finite-variations, however, its volatility is an increasing process that measures the dissipation of the main-system flow (energy, biomass, information). A finite-variation process is given as a difference of two increasing functions, providing a suitable representation for the birth and death of biomass in biology. Its predictability depends on the history introduced in the stochastic basis that represents the phenomena. 

Here, we first explain our concept of fitness and its volatility in the case of discrete-time evolutions through the concept of predictable projection of a stochastic process. Then, using adequate changes in time and space scales, we analyzed diffusion processes appearing as a limit in the distribution of rescaled discrete-time processes.

We considered the class of homogeneous stochastic differential equations for these processes to explore a central qualitative feature of the dynamics: the existence of equilibrium. This analysis proceeded by recovering Gibbs-type equilibrium probability densities and connecting fitness and its volatility to free energy and the scale function. The resulting characterization, under neutrality \cite{hubbell2011unified}, is that fitness is an affine function of biomass if and only if the equilibrium measure is of Dirichlet type, which is in agreement with previous work on genetics and ecology showing that under neutrality the proportional abundance of species converges to a Beta distribution, in the one dimensional case, and Dirichlet in the multidimensional one \cite{marquet2017proportional, videla2025persistence}. 

In our previous research \cite{videla2025persistence}, we investigated the evolution of biomass proportions on a $d$-dimensional simplex using a geometric analysis based on Shahashahani's metric. That allowed us to define a concept of fitness potential whose Shahashahani gradient represented the main system dynamics in the stochastic differential model. It is well known that Shahashahani's metric is equivalent to Fisher's information for two elements $p,q$ in the simplex. In the current contribution, we go beyond the simplex, recovering the notion of fitness potential through the analysis of equilibrium via the entropy-production and probability-current apparatus (Section 3.2) inspired in Gibbs-type distributions. 

Absolute-abundance models, characterized by a stationary Gamma invariant distribution in Section 4, are well-known  in stochastic population dynamics \cite{lande2003stochastic}, although their treatment through the general predictable/volatility decomposition of Section 2 is, to our knowledge, new.

The equilibrium analysis via entropy production and probability currents also makes an explicit historical connection. The entropy-production functional computed in Section 3.2 (eq. 3.26), and the Gibbs-type equilibrium densities of Section 4.2, sit in the same conceptual lineage as Fisher's Fundamental Theorem of Natural Selection (FTNS), which ties the rate of increase of mean fitness to the additive genetic variance in fitness within a population. Fisher's own motivation for the variance-increase connection was an analogy, informal at the time, to the increase of entropy under Boltzmann's second law, decades before Shannon gave entropy its modern interpretation as an information measure \cite{smith2023beyond}, and a fully rigorous restatement shows the FTNS's fitness variance to be exactly the speed of the population's probability distribution under the Fisher-information metric \cite{baez2021fundamental}. This lineage has since been made considerably more concrete in the forms directly relevant to Sections 3–4. Sella and Hirsh \cite{sella2005application} show that the stationary distribution of a Markov chain for genotype-fixation under selection and drift is exactly of the Gibbs-Boltzmann form, the population size playing the role of the inverse temperature, and construct an explicit Lyapunov function of free-energy-type or ``free fitness” that reduces to FTNS in the infinite-population limit and picks up an explicit entropy (drift) term at a finite population size. Barton and Coe \cite{barton2009application} extended this to continuous allele-frequency diffusions: Wright's stationary distribution is shown to be Boltzmann-distributed, again with population size playing the role of inverse temperature, generalized beyond the FTNS's additive-only scope to arbitrary epistasis and dominance via an additional entropy term counting the volume of allele-frequency configurations consistent with an observed trait. This is the closer mathematical cousin of our own Theorem 2, independently derived in a population-genetics idiom for essentially the same class of diffusion; extending Theorem 2 to accommodate heterogeneous fitness across species is a natural route toward the non-neutral generalization discussed below. A complementary, more dynamical generalization comes from Mustonen and Lässig \cite{mustonen2010fitness}, whose “fitness-flux theorem” extends the FTNS to mutation, genetic drift, and time-dependent selection via a fluctuation-theorem argument, a non-negative average flux rather than a static equilibrium form. Because fitness flux is a probability-current construction rather than an equilibrium-distribution one, it is arguably the closer conceptual analog of our own entropy-production functional (eq. 3.26), which is likewise a current rather than a snapshot; making that correspondence precise, alongside the equilibrium-side correspondence to \cite{barton2009application, sella2005application}, is a natural target for future work. Our $V^p (X)$ remains a distinct object from any of these, pathwise and temporal rather than cross-sectional, and whether a precise formal correspondence exists is an open question we do not resolve here, but the entropy-production formalism already in Section 3.2 seems like the natural place to look for one.

Two interesting properties of our framework are worth noticing. First, the homogeneous diffusion describing the dynamics of relative biomass abundance on the simplex has a remarkable property that we stressed in section 3: under the postulate of the Neutral Theory, fitness is an affine function if and only if the equilibrium measure is of Dirichlet type. And secondly, the space $L^2$ associated with Gibbs type measures has another interesting property: the existence of an orthonormal basis of orthogonal polynomials. As a result, these polynomials play an important role in statistical data analysis. Here, we dedicated section 4 to show three important model examples: the so called Jacobi, Cox-Ingersoll-Ross, and stochastic logistic models.

The main results first show the importance of characterizing fitness as a couple of predictable processes. This allows us to clarify the relationship between the focal or main ecological dynamics and its relation to the environment, represented by a martingale. As a result, our paper stresses the relationship between fitness and its volatility to ensure the existence of an equilibrium in the case of homogeneous dynamics of the biomass. In other words, this result allows for a clearer exploration of persistence. 

Further research will focus on non-homogeneous diffusion models as well as more general dynamics, such as those driven by Lévy processes. 

A recent line of work \cite{smith2023beyond} asks whether any single summary statistic called “fitness” can be dynamically sufficient (sufficient, that is, to predict its own future evolution) once the dynamics of a population depend on the relationships between its members,the mating structure, the stages of life history, interaction networks, frequency dependence, and other components of the total genetic variance not taken into account, and that a fitness function, by construction, discards. The notion of dynamical sufficiency itself originates from Lewontin \cite{lewontin1974genetic}, as a criterion for the adequacy of a statistically inferred model class to predict the dynamics to which it is fitted to; Smith \cite{smith2023beyond} develops this criterion into a formal critique of additive fitness using population genetics with explicit lifecycle transitions and showing that additive fitness statistics generally capture only the partial change of a character, leaving a residual that depends on how relations are reconstructed each generation and that fitness alone cannot represent. We agree with Smith in that this is a problem in current definitions of fitness, one that our proposal sidesteps by working with a general biomass flow, which can be complicated by including interactions and other aspects of the life histories of organisms that ultimately affect birth and death rates in interaction with the environment represented by a martingale or noise, thus fitness becomes a predictable biomass flux and its environment, similar to the notion of fitness as a population-environment couple \cite{abrams2023evolution}. Ours is a general mathematical definition of fitness as a stochastic predictable process and its predictable volatility that can be applied across levels of biological organization, aggregation, and complexity as long as these can show a birth-and-death dynamics.

\vskip 20pt
\textbf{Acknowledgments.} The authors thank Dr. Rodrigo Plaza for kindly providing the simulation and parameter estimation codes for the Jacobi process.

\textbf{Declaration:} The authors confirm the use of artificial intelligence exclusively for bibliographic research and language editing.
\vskip 20pt
\bibliographystyle{plain}
\bibliography{references}
\end{document}